\documentclass[12pt]{amsproc}
\usepackage{marginnote}
\usepackage{mathrsfs}
\usepackage{amsmath,amssymb,amsthm,setspace,mathtools,enumerate,comment}
\usepackage{blindtext}
\usepackage{subfiles} 
\usepackage{textcomp}
\usepackage{marginnote}
\usepackage[myheadings]{fullpage}

\usepackage[backref=page]{hyperref}
\usepackage{colonequals}
\hypersetup{
    colorlinks=true,
    citecolor=red!80!blue, 
    linkcolor=blue,
    }    

\usepackage[capitalise, nameinlink]{cleveref}
\usepackage{graphicx}
\usepackage{tikz,tikz-cd}
\usetikzlibrary{arrows}
\usetikzlibrary{shapes}

\newtheoremstyle{theoremdd}
  {10pt}
  {10pt}
  {\itshape}
  {}
  {\sc}
  {. }
  { }
  {\thmname{#1}\thmnumber{ #2}\thmnote{ (#3)}}
\newtheoremstyle{definitiondd}
  {10pt}
  {10pt}
  {}
  {}
  {\sc}
  {. }
  { }
  {\thmname{#1}\thmnumber{ #2}\thmnote{ (#3)}}

\theoremstyle{theoremdd}
     \newtheorem{theorem}{Theorem}[section]
     \newtheorem{conjecture}{Conjecture}[section]
     \newtheorem{lemma}[theorem]{Lemma}
     \newtheorem{proposition}[theorem]{Proposition}
    \newtheorem{corollary}[theorem]{Corollary}
\theoremstyle{definitiondd}
    \newtheorem{definition}[theorem]{{Definition}}
    \newtheorem{definitions}[theorem]{{Definitions}}
    \newtheorem{question}{Question}

    \newtheorem{example}[theorem]{{Example}}
    
    \newtheorem{examples}[theorem]{Examples}
    \newtheorem{remark}[theorem]{{Remark}}
\theoremstyle{remark}
    \newtheorem*{remarks}{{Remarks}}

\newcommand{\F}{\ensuremath{{\mathbb{F}}}}

\newcommand{\Tr}{\ensuremath{{\mathrm{Tr}}}}
\newcommand{\Hom}{\ensuremath{\mathrm{Hom}}}

\newcommand{\m}{\ensuremath{\mathfrak{m}}}

\newcommand{\Z}{\ensuremath{\mathbb{Z}}}
\newcommand{\Q}{\ensuremath{\mathbb{Q}}}

\newcommand{\GL}{\ensuremath{\mathrm{GL}}}

\newcommand{\rhobar}{\ensuremath{\overline{\rho}}}

\newcommand{\End}{\ensuremath{\mathrm{End}}}

\newcommand{\semis}{\ensuremath{\mathrm{ss}}}

\newcommand{\GSp}{\ensuremath{\mathrm{GSp}}}

\newcommand{\SL}{\ensuremath{\mathrm{SL}}}

\newcommand{\scrS}{\ensuremath{\mathscr{S}}}
\newcommand{\JH}{\ensuremath{\mathrm{JH}}}
\newcommand{\Lambdabar}{\ensuremath{\bar{\Lambda}}}

\newcommand{\CalG}{\ensuremath{\mathcal{G}}}

\title{Symplectic extension graphs and Ribet's lemma}
\author{Aniruddha Sudarshan}
\address{Department of Mathematics, Temple University, Philadelphia, PA 19122,
USA}
\email{aniruddha.sudarshan@temple.edu}
\usepackage{microtype}
\begin{document}

    \begin{abstract}
  Let $p$ be an odd prime, and let $\rho\colon G\to \mathrm{GSp}_{2n}(\mathbb Z_p)$ be an irreducible symplectic residually multiplicity-free representation.
We define the symplectic extension graph $\Gamma_\rho^{\mathrm{GSp}}$ associated with $\rho$ to be the directed graph whose vertices are the residual factors of $\rho$ and a directed edge between two vertices corresponds to a nonsplit extension between the vertices realized by a $G$-stable symplectic lattice. Using the Euclidean Bruhat--Tits building associated with $\mathrm{GSp}_{2n}(\mathbb Q_p)$, we prove, under suitable assumptions, the existence of certain paths in $\Gamma_\rho^{\mathrm{GSp}}$ and show that $\Gamma_\rho^{\mathrm{GSp}}$ is a disjoint union of nonempty strongly connected subgraphs. Moreover, we provide examples to show that the result we obtain is optimal.
\end{abstract}
	\maketitle
    \tableofcontents

\section{Introduction}
    In his remarkable paper, Ribet proves the converse to Herbrand's theorem by using certain congruences between modular forms \cite{ribetunram}. A crucial step in his proof is the construction of a nonsplit extension between characters, usually called \textit{Ribet's lemma}  \cite[Proposition 2.1]{ribetunram}. Throughout this paper, we fix an odd prime $p$.
    Let $V$ be a $2$-dimensional $\Q_p$-vector space,  $\rho \colon G\to \GL(V)$ be an irreducible representation that admits a $G$-stable $\Z_p$-lattice in $V$, and let $\rho$ be residually reducible with factors $\chi_1$ and $\chi_2$. Ribet shows the existence of $G$-stable $\Z_p$-lattices $\Lambda_1$ and $\Lambda_2$ of $V$ such that $\Lambda_i/p\Lambda_i$ is a nonsplit extension of $\chi_i$ by $\chi_j$ for all $1\leq i\neq j\leq 2$. 

    The sketch of a modern proof of Ribet's lemma is as follows.
    The $p$-adic Galois representation associated with a cuspidal elliptic Hecke eigenform that is congruent to an Eisenstein series modulo $p$ satisfies Ribet's hypothesis. 
    Using the local properties (such as \textit{ordinary} at $p$) of such a Galois representation, Ribet constructs nontrivial unramified cohomology classes.
    This method of using congruences between modular forms to construct nontrivial elements in certain Selmer groups is often called \textit{Ribet's method}.

\subsection{Bella\"iche's extension graph}
Bella\"iche attributes to Serre a proof of \cite[Proposition 2.1]{ribetunram} using the Bruhat--Tits tree associated with $\GL_2(\Q_p)$ \cite{BellaicheRibet}.
Generalizing this method, for all $n\geq 2$, Bella\"iche obtains a $\GL_n$-version of Ribet's lemma using the Bruhat--Tits building associated with $\GL_n(\Q_p)$ \cite[Theorem 1]{BellaicheRibetslemmaPadova}.

Let $V$ be a finite dimensional vector space over $\Q_p$, and let $\rho \colon G\to \GL(V)$ be a $p$-adic representation that admits a $G$-stable $\Z_p$-lattice in $V$. For example, this occurs when $G$ is compact and $\rho$ is continuous \cite[Corollary 1.1]{BellaicheRibet}. For a $G$-stable lattice $\Lambda$, let $\rhobar_\Lambda$ be the residual representation $\Lambda/p\Lambda$. Although the representation $\rhobar_\Lambda$ depends on $\Lambda$, its semisimplification $\rhobar^{\mathrm{ss}}_\Lambda$ is independent of the choice of $\Lambda$ \cite[Proposition 1.2]{BellaicheRibet}.
        The \textit{residual factors} of a representation $\rho\colon G\to \GL(V)$ that admits a $G$-stable lattice $\Lambda$ are defined to be the irreducible Jordan--H\"older factors of $\bar{\rho}_\Lambda^{\mathrm{
        ss
        }}$. We assume the residual factors of $\rhobar$ are absolutely irreducible.  Let $\rho_1,\dots, \rho_d$ be the isomorphism classes of the residual factors of $\rho$.

\begin{definition}
    The \textit{extension graph} $\Gamma_\rho^{\GL}$ is the directed graph on the set of vertices  $\{\rho_1,\dots,\rho_d\}$, where an edge from $\rho_i$ to $\rho_j$ exists if there exists a $G$-stable lattice $\Lambda$ in $V$ such that $\rhobar_\Lambda$ admits a nonsplit extension of $\rho_i$ by $\rho_j$ as a subquotient. 
\end{definition}
A directed graph is \textit{strongly connected} if there is a directed path between any two vertices.

\begin{theorem}\label{prop:Bellaichemain}\cite[Theorem 1]{BellaicheRibetslemmaPadova}
    If $\rho$ is irreducible, then $\Gamma_\rho^\GL$ is strongly connected. 
\end{theorem}
Under the assumption that each $\rho_i$ occurs with multiplicity one in $\rhobar^{\mathrm{ss}}$, Bella\"iche--Graftieaux show that $\rho$ is irreducible if and only if $\Gamma_\rho^\GL$ is strongly connected \cite[Proposition 3.4.1]{BellaicheGraftieaux}. Thus, \cref{prop:Bellaichemain} is optimal under the given hypothesis. 
When $d=2$, \cref{prop:Bellaichemain} is a reformulation of \cite[Proposition 2.1]{ribetunram}. 
Using Ribet's method, \cref{prop:Bellaichemain} has some arithmetic applications (for example, \cite{BellaicheChenevierU3} and  \cite{bergerklosinirred}).  
\subsection{Our results}
For an irreducible symplectic representation, we
define and study the subgraph of Bella\"iche's extension graph whose edges correspond to nonsplit extensions realized by stable symplectic lattices. 

Let $V$ be a $\Q_p$-vector space with a symplectic (see \cref{def:sympairing}) bilinear form $\phi$. 
Let $\GL(V,\phi)$ be the subgroup of $ \GL(V)$ consisting of elements $M$ such that,  \[\phi(Mv,Mw)=\lambda(M)\phi(v,w)\quad \textnormal{for all }v,w\in V\textnormal{ and for some }\lambda(M)\in \Q_p^\times.\] A representation $\rho \colon G\to \GL(V)$ is \textit{symplectic} if $\rho(G)\subseteq \GL(V,\phi)$. 

A $\Z_p$-lattice $\Lambda$ in $V$ is a \textit{symplectic lattice} if the bilinear form on $\Lambda/p\Lambda$ induced\footnote{See the discussion before \cref{def:symplattice}} by $\phi$ is symplectic. If $\Lambda$ is symplectic and $\phi(\Lambda\times \Lambda)=\Z_p$, then we call $\Lambda$ \textit{primitive}. 
Let $\rho \colon G\to \GL(V,\phi)$ be an irreducible symplectic representation that admits a $G$-stable primitive lattice. We note that $\rho$ stabilizes a primitive lattice if and only if there exists a (symplectic) basis of $V$ with respect to which $\rho(G)\subseteq \GSp_{2n}(\Z_p)$. We assume the residual factors of $\rhobar$ are absolutely irreducible. Let $\rho_1,\dots, \rho_d$ be the isomorphism classes of the residual factors of $\rho$. 

\begin{definition}
    The \textit{symplectic extension graph} $\Gamma_\rho^{\GSp}$ is the directed graph on the set of vertices $\{\rho_1,\dots,\rho_d\}$ such that an edge from $\rho_i$ to $\rho_j$ exists if there exists a $G$-stable symplectic lattice $\Lambda$ in $V$ such that $\rhobar_\Lambda$ admits a nonsplit extension of $\rho_i$ by $\rho_j$ as a subquotient.  
\end{definition}

We show that $\Gamma_\rho^{\GSp}$ is a skew-symmetric graph (see \cref{def:skewsymgraph}). 
Under further assumptions on the residual factors of $\rho$, we show the existence of certain paths in $\Gamma_\rho^{\GSp}$ and show that  $\Gamma_\rho^{\GSp}$ is a disjoint union of strongly connected graphs.

To state our results precisely, we introduce some notation. 
For all $g\in G$, we let $\lambda_\rho(g)\colonequals\lambda(\rho(g))$. Since $\rho$ admits a $G$-stable primitive lattice, we see that $\lambda_\rho(G)\subseteq  \Z_p^\times$. Let $\tau$ be the $\Z_p$-linear involution of the group ring $\Z_p[G]$ defined on $G$ by $\tau(g)=\lambda_\rho(g)g^{-1}$. For a matrix $M$, we let $M^t$ be its transpose. For a representation $\psi\colon G\to \GL_{2n}(\Z_p)$, we let $\psi^\perp\colonequals (\psi\circ \tau)^t$.
Since $\rho(G)\subseteq \GL(V,\phi)$, there exists a involution $\sigma$ on $I=\{1,\dots,d\}$ such that
\[
    \rho_{\sigma(i)}\simeq \rho_i^\perp\quad \textnormal{for all }i\in I.    
\] 
\begin{definition}\cite{skewsymmetricgraphs}\label{def:skewsymgraph}
    A directed graph $(\mathcal{V},E)$ is \textit{skew symmetric} if there is a involution $\sigma$ on $\mathcal{V}$ such that  $\sigma(u)\neq u$ for all $u\in \mathcal{V}$ and, for any $u,v\in \mathcal{V}$,
         $(u,v)\in E$ if and only if $(\sigma(v),\sigma(u))\in E$.
\end{definition} 
\begin{proposition}\label{prop:skeqsymgraph}
Let $\rho\colon G\to \GL(V,\phi)$ be an irreducible symplectic $p$-adic representation stabilizing a primitive lattice in $V$. Let $\rho_1,\dots,\rho_d$ be the residual factors of $\rho$ that are pairwise nonisomorphic. 
    Let $\sigma$ be the involution on $\{1,\dots,d\}$ defined by $\rho_{\sigma(i)}\simeq \rho_i^\perp$. Suppose that  $\sigma(i)\neq i$ for all $1\leq i\leq d$.
    Then the directed graph $\Gamma_\rho^{\GSp}$ is a skew-symmetric graph with respect to the involution on $\{\rho_1,\dots,\rho_d\}$  induced by $\sigma$. 
\end{proposition}

Throughout the paper, we usually assume 
\begin{equation}\tag{rmf}\label{assume:resmultfree}
    \rho \textnormal{ is residually multiplicity free.}
\end{equation}
For our main results, we also need an extra assumption:
\begin{equation}\label{assume:dist}
    \tag{dist} 
    \textnormal{There exists $g_0\in G$ such that $\rhobar(g_0)$ has distinct eigenvalues.}
\end{equation}

We now state the main theorem of the paper. All the results we state for $p$-adic representations hold, with the same proof, when $\Z_p$ is replaced by a complete discrete valuation ring with finite residue field of characteristic $p$. 
\begin{theorem}\label{thm:ourmain}
Let $(V,\phi)$ be a symplectic $\Q_p$-vector space.
Let $\rho \colon G\to \GL(V,\phi)$ be an irreducible symplectic representation that admits a $G$-stable primitive lattice, with absolutely irreducible residual factors $\rho_1,\dots, \rho_d$. Let $\sigma$ be the permutation on $\{1,\dots,d\}$ defined by $\rho_{\sigma(i)}\simeq \rho_i^\perp$.
Assume that $\rho$ satisfies \eqref{assume:dist} and \eqref{assume:resmultfree}, and that $\sigma(i)\neq i$ for all $1\leq i\leq d$. 
  Then the directed graph $\Gamma_\rho^{\GSp}$ satisfies the following conditions.
    \begin{enumerate}[$(a)$]

        \item \label{item:main2} For all $1\leq i\leq d$, $\rho_i$ and $\rho_{\sigma(i)}$ lie in the same strongly connected component.     
        \item \label{item:main1} It is a disjoint union of  at most $d/2$ strongly connected graphs.  
    \end{enumerate}
    \label{item:main3} In particular, if $\Gamma_\rho^{\GSp}$ has $d/2$ strongly connected components, then $\Gamma_\rho^{\GSp}$ is the disjoint union of strongly connected graphs with vertices $\{\rho_i,\rho_{\sigma(i)}\}_{i<\sigma(i)}$.
\end{theorem}

When $d=2$, 
\cref{thm:ourmain} is a generalization of \cite[Proposition 1.1]{UrbanKlingen} (Urban works over $\GSp_4$) for complete discrete valuation rings with finite residue field of odd characteristic. Using \cite[Proposition 1.1]{UrbanKlingen}, Urban produces nontrivial elements in a Selmer group of certain twists of $\mathrm{ad}^0$ of a modular $p$-adic Galois representation \cite[Theorem 3.7]{UrbanKlingen}.

We show that \cref{thm:ourmain} is the optimal result one could achieve in this generality.

\begin{theorem}
    \label{thm:examples}
    Given a positive integer $n$, and an integer $r\in \{1,\dots, n\}$, there exists a group $G_r$ and an irreducible symplectic representation $\rho\colon G_r\hookrightarrow \GSp_{2n}(\Z_p)$ satisfying \eqref{assume:dist} and \eqref{assume:resmultfree} such that 
    \begin{enumerate}[$(a)$]
        \item  the residual factors of $\rho$ are characters,
        \item $\Gamma_\rho^{\GSp}$ is a disjoint union of $r$ strongly connected graphs.
    \end{enumerate}
\end{theorem}

We believe \eqref{assume:dist} should be omittable, and should be equivalent to \eqref{assume:resmultfree} (this is the case when $n=1$). On the other hand, the assumption $\sigma(i)\neq i$ for all $1\leq i\leq d$ is necessary for \cref{thm:ourmain} and the method of proof. 
To be precise, we produce a symplectic irreducible $p$-adic representation $\rho$  that is residually reducible with three absolutely irreducible residual factors such that $\Gamma_\rho^{\GSp}$ is an empty graph but $\Gamma_\rho^\GL$ is a complete graph (see \cref{prop:example3JH}). 

\subsubsection{The work of Bella\"iche--Graftieaux} In proving our main results, we prove symplectic versions of certain results in \cite{BellaicheGraftieaux}. In this subsection, we give an overview of some crucial results from \cite{BellaicheGraftieaux}.  

Let $V$ be an $n$-dimensional $\Q_p$-vector space and let $\rho\colon G\to \GL(V)$ be an irreducible representation that admits a $G$-stable $\Z_p$-lattice. Suppose that $\rho$ satisfies \eqref{assume:resmultfree} and $\rho_1,\dots,\rho_d$ are its absolutely irreducible residual factors. Let $S_\rho^\GL$ be the (finite) set of homothety classes of $G$-stable lattices in the Bruhat--Tits building associated with $\GL(V)$. Then there exists a basis $\{v_1,\dots,v_n\}$ of $V$ such that every vertex in $S_\rho^\GL$ is of the form $[a_1,\dots,a_n]\colonequals [\langle p^{a_1}v_1,\dots,p^{a_n}v_n\rangle]$ for some integers $a_1,\dots,a_n$ \cite[Proposition 3.1.1]{BellaicheGraftieaux} (see \cref{prop:Srhosymliesinanapartment} for the analogous symplectic version). Moreover, there exists a \text{metric} $\eta\colon \{1,\dots,n\}^2\to \Z$ such that 
\[
    S_\rho^\GL=\{[a_1,\dots,a_n] \ |\ a_i-a_j\leq \eta(i,j)\}.
\]

In fact, $S_\rho^\GL$ lives in a sub-apartment of dimension $d-1$: there exists a partition $\{1,\dots,n\}=\bigsqcup_{j=1}^d B_j$, such that $\eta=0$ on $B_j\times B_j$ for all $1\leq j\leq d$. This implies that $\eta$ is constant on $B_j\times B_k$ for all $1\leq j,k\leq d$. Thus,  $\eta$ corresponding to $S_\rho^\GL$ defines a \text{metric graph} $\mathcal{G}(S_\rho^\GL)$ on $\{1,\dots,d\}$. Conversely, given a metric graph $\mathcal{G}$ on $\{1,\dots,d\}$, there exists a group $G$ and an irreducible representation $\rho\colon G\to \GL_n(\Z_p)$ such that $\mathcal{G}(S_\rho^\GL)=\mathcal{G}$ \cite[Lemme 3.2.1]{BellaicheGraftieaux}. The proof of \cref{thm:examples} is inspired by the proof of  \cite[Lemme 3.2.1]{BellaicheGraftieaux}.


\subsection{Method of proof} 
Let $V$ be a finite-dimensional $\Q_p$-vector space with a symplectic form $\phi$. 
Let $\rho\colon 
G\to \GL(V,\phi)$ be an irreducible symplectic representation that admits a $G$-stable primitive lattice and satisfies \eqref{assume:resmultfree}. Let $\rho_1,\dots,\rho_d$ be the absolutely irreducible residual factors of $\rho$. Let $I=\{1,\dots,d\}$, $\sigma\colon I\to I$ be defined by $\rho_{\sigma(i)}\simeq \rho_i^\perp$, and let $I_0\colonequals \{i\in I\ |\ \sigma(i)=i\}$.
 
We prove \cref{thm:ourmain} by using the Bruhat--Tits building $\Delta$ associated with $\GL(V,\phi)$. The building $\Delta$ is a simplicial complex whose vertices are either homothety classes $[\Lambda_0]$ of primitive lattices, or correspond to isotropic subspaces of $\Lambda_0/p\Lambda_0$. Let $\scrS_\rho^\GSp$ be the simplicial subcomplex of $\Delta$ formed by $G$-stable vertices in $\Delta$. Since $\rho$ is irreducible, $\scrS_\rho^\GSp$ is compact. 
The vertices on the boundary of $\scrS_\rho^\GSp$ correspond to indecomposable residual representations. 
Under the assumption $I_0=\varnothing$, we show the existence of certain symplectic vertices on the boundary of $\scrS_\rho^\GSp$. These are precisely the vertices that realize certain edges in the symplectic extension graph associated with $\rho$. 

The assumption $I_0=\varnothing$ is equivalent to the existence of neighbouring symplectic vertices in $\scrS_\rho^\GSp$ (see \cref{thm:symplecticneighboursiffI0empty}). We define the Jordan--H\"older colouring on the edges in $\scrS_\rho^\GSp$ as in \cite{BellaicheRibetslemmaPadova}. Assuming $I_0=\varnothing$ and \eqref{assume:dist}, we can move between symplectic vertices in  $\scrS_\rho^\GSp$ to show the existence of $G$-stable symplectic vertices (on the boundary) with no edges of a specific colouring (see \cref{prop:existssymplecticwithnoneighbourswithspecificJH}). This is equivalent to the existence of edges in $\Gamma_\rho^\GSp$ (see \cref{prop:equivnonexistneighexitstext}). 

The proof of \cref{thm:examples} builds on the examples constructed when $d\in \{2,4\}$. 
To prove \cref{thm:examples}, we use a construction inspired by \cite[Lemme 3.2.1]{BellaicheGraftieaux}. Given a positive integer $r\leq n$, we define a metric $\eta^{(r)}$ on $I$ such that $G_r\colonequals \{(a_{ij})\in \GSp_{2n}(\Z_p)\ |\ a_{ij}\in p^{\eta^{(r)}(i,j)}\Z_p\}$ with the inclusion $\rho\colon G_r\hookrightarrow  \GSp_{2n}(\Z_p)$ satisfies the conclusions of \cref{thm:examples}. A similar construction is used to produce an example of a $p$-adic representation whose symplectic extension graph is empty (see \cref{prop:example3JH}).

\subsection{Motivation and future direction} 
Ribet's lemma was generalized to prove certain cases of the Iwasawa main conjectures \cite{Mazur--Wiles,WilesTotallyReal,UrbanKlingen,SkinnerUrban}. 
More recently, Dasgupta--Kakde and collaborators proved the Brumer--Stark conjecture by generalizing Ribet's lemma for two-dimensional representations valued in complete reduced local rings \cite{SamitKakde,dasgupta2023brumerstarkconjecturez,dasgupta2023residually}.

Bella\"iche--Chenevier generalize Ribet's lemma to local henselian rings by use of \textit{pseudorepresentations} and \textit{generalized matrix algebras} (GMAs) \cite[\S 1]{BellaicheChenevierFamilies}.   
Given an irreducible pseudocharacter that is residually multiplicity free, they provide a systematic way to construct nonsplit extensions between residual representations \cite[Theorem 1.5.5]{BellaicheChenevierFamilies}.
As an arithmetic application, they construct more than one linearly independent element in certain Selmer groups to prove cases of the Bloch--Kato conjecture \cite[Chapter 5]{BellaicheChenevierFamilies}.

Recently, Moakher--Quast have developed an analogous theory of symplectic\textit{ determinants} (in the sense of \cite{ChenevierDeterminants}) and symplectic GMAs \cite{moakher2023symplecticdeterminantlawsinvariant}.
 Building on this new theory, one hopes for a symplectic analogue, if it exists, of \cite[Theorem 1.5.5]{BellaicheChenevierFamilies} and its possible arithmetic applications. As a first step in this pursuit, we study extensions of representations realized by symplectic lattices.

\subsection{Outline of the paper} In \cref{sec:background}, we fix notation and provide background on symplectic modules and representations.
In \cref{sec:symtype}, we recall a few results and definitions from \cite[\S 1.8]{BellaicheChenevierFamilies} in our setting and prove \cref{prop:skeqsymgraph}. The key result of this section is that, given a symplectic residually multiplicity-free $p$-adic representation $\rho$, under some assumptions, we show the existence of a ``special" symplectic basis with respect to which the diagonal matrices lie in the image of $\rho$. 

In \cref{sec:build}, we introduce Bruhat--Tits buildings associated with general linear groups and symplectic groups over $\Q_p$. A building is an abstract simplicial complex that is a union of subcomplexes called apartments. By their description, an apartment corresponds to a basis of the underlying vector space and is easier to work with. For example, moving between vertices in an apartment corresponds to conjugation of lattices by certain diagonal matrices.
For a symplectic irreducible $p$-adic representation $\rho$, we study the geometry of $\scrS_\rho^{\GSp}$, the simplicial complex formed by $G$-stable lattices in the symplectic building. We show that $\scrS_\rho^{\GSp}$ lies in a symplectic apartment corresponding to a special symplectic basis as in \cref{sec:symtype}. This gives us all the data of vertices and their neighbours in $\scrS_\rho^\GSp$ with respect to the symplectic basis, and is used to prove the existence of certain boundary vertices of $\scrS_\rho^\GSp$.   

In \cref{sec:symextgraph}, using the main results from \cref{sec:build}, we prove our main results \cref{thm:ourmain} and \cref{thm:examples}. We end with an example of an irreducible symplectic representation whose extension graph is empty (see \cref{prop:example3JH}).

\subsection*{Acknowledgement} I thank my advisor, Jaclyn Lang, for continuous support and encouragement throughout writing this article. I  also thank Mohamed Moakher for useful conversations and comments on an earlier draft, Matthew Stover for conversations on buildings, and Carl Wang-Erickson for helpful suggestions. The author was partially supported by the National Science Foundation grant DMS-2301738 while completing this work. 

\subsection*{Statement of AI use} I did not use generative AI tools for research and writing of this paper. I used the free versions of ChatGPT and Claude for fixing minor typos.

\section{Symplectic modules and representations}
    \label{sec:background}
In this section, we give a brief introduction to symplectic modules and representations. 
\subsection{Symplectic modules}
We recall a few facts about finite free symplectic modules over a (commutative) ring $A$ with $2\in A^\times$. 
Given an $A$-basis $\{e_1,\dots, e_{n}\}$ of $V$, the matrix representing $\phi$ on $V$ is the $n\times n$ matrix whose $(i,j)$-th entry is $\phi(e_i,e_j)$. 

\begin{definition}\label{def:sympairing}
     A \textit{symplectic pairing} on a finite free $A$-module $V$ is a  bilinear pairing $\phi\colon V\times V\to A$ satisfying
\begin{enumerate}
    \item \label{item:symplecticpairing1} (alternating) $\phi(v,v)=0$ for all $v\in V$;
    \item \label{item:symplecticpairing2} (perfect) the map $V\to \Hom_A(V,A)$ mapping $v$ to the $A$-linear map $w \mapsto \phi(v,w)$ is an isomorphism of $A$-modules.
\end{enumerate}
\end{definition}

\begin{remarks}\quad 
    \begin{enumerate}[(a)]
        \item \label{rem:skewsymiffsympair} A bilinear pairing $\phi$ is a symplectic pairing on $V$ if and only if, for any given basis of $V$, the matrix $Q\in M_{n}(A)$ representing $\phi$ is invertible and skew symmetric; that is, $Q\in \GL_{n}(A)$ and $Q^t=-Q$. 
Using this, one can check that a symplectic pairing on a finite free $A$-module $V$ exists if and only if $V$ has even rank.
\item \label{rem:pairing=1}\cref{def:sympairing}\eqref{item:symplecticpairing2} is stronger than just $\phi$ being nondegenerate. To be precise, if $\phi$ is a symplectic pairing on a finite free module $V$ over a ring $A$, then for every nonzero $w\in V$ that can be extended to an $A$-basis of $V$, there exists $v\in V$ such that $\phi(v,w)=1$.
\item When $A$ is a field, \cref{def:sympairing}\eqref{item:symplecticpairing2} is equivalent to $\phi$ being nondegenerate. 
    \end{enumerate}
\end{remarks}

By the notion of a symplectic pairing, we make the following definitions as in \cite[\S 1]{LamSwanSymMod}.
\begin{definitions}
    Let $A$ be a ring with $2\in A^\times$.
    A \textit{symplectic module} over $A$ of rank $2n$ is a tuple $(V,\phi)$  where $V$ is a finitely generated free
    module over $A$ with rank $2n$ and $\phi$ is a symplectic bilinear pairing on $V$. 
A submodule $W$ of $V$ is said to be \begin{enumerate}[(a)]
        \item \textit{isotropic} if $\phi|_{W\times W}=0$;
        \item \textit{symplectic} if $\phi|_{W\times W}$ is perfect. 
    \end{enumerate}
\end{definitions}
A $\Z_p$-lattice $\Lambda$ in a symplectic $\Q_p$-vector space $(V,\phi)$ is primitive (that is, $\phi(\Lambda\times \Lambda)=\Z_p$ and $\phi|_{\Lambda\times \Lambda} \bmod{p}$ is symplectic on $\Lambda/p\Lambda$) if and only if $(\Lambda,\phi|_{\Lambda\times \Lambda})$ is a symplectic $\Z_p$-module. It follows by definition that a symplectic $\Z_p$-module is primitive. We prove the converse later: a primitive $\Z_p$-lattice admits a symplectic $\Z_p$-basis (see \cref{def:symbasis} and \cref{cor:primlatticeissympmod}).

\begin{definition}
    Let $(V,\phi)$ be a symplectic $A$-module.
    For a subset $W$ of $V$, we write $W^0$ for the submodule of $V$ consisting of elements $w^0\in V$ such that $\phi(w^0,w)=0$ for all $w\in W$.
\end{definition}
We observe that $W$ is isotropic if and only if $W\subseteq W^0$. If $(W,\phi|_{W\times W})$ is a symplectic $A$-submodule of $V$, then $W\cap W^0=0$. The converse holds when $A$ is a field, but is not true generally since $W\cap W^0=0$ is equivalent to $\phi|_{W\times W}$ being nondegenerate.

\begin{definition}
    An isotropic subspace $W$ of $V$ is \textit{Lagrangian} if it satisfies $W=W^0$. In other words, a Lagrangian subspace is a maximal isotropic subspace of $V$.
\end{definition}


\begin{proposition}\label{prop:sympbasis}
    Let $A$ be a local ring with maximal ideal $\m$ and $2\in A^\times$. Let $V$ be a free module of rank $2n$ over $A$, and let $\phi$ be an alternating bilinear pairing on $V$. Then $\phi$ is a symplectic pairing if and only if, given an element $w_1\in V\setminus \m V$,
    there is a basis of $V$ containing $w_1$ with respect to which the matrix representing $\phi$ is  
    \[
        J_{2n}\colonequals\begin{pmatrix}
        &-I_n\\I_n&
    \end{pmatrix},
    \]
    where $I_n$ is the identity matrix in $\GL_n(A)$.
\end{proposition}
\begin{proof}
The \textit{if} part is immediate. We now prove the \textit{only if} part. 
By the definition of a symplectic pairing (see Remark \eqref{rem:pairing=1}), there exists $v_1\in V$ such that $\phi(v_1,w_1)=1$. If $n=1$, then $V$ is the span of $\{w_1,v_1\}$ and the result follows. Let $W$ be the span of $\{w_1,v_1\}$. Then it is easy to see that $V=W\oplus W^0$ and $W^0$ is a symplectic submodule of $V$ of dimension $2n-2$. The proof follows by induction on $n$. 
\end{proof}
When the dimension $n$ is implicit, we write $J$ and $I$ instead of $J_{2n}$ and $I_n$, respectively.
\begin{definition}\label{def:symbasis}
    Let $(V,\phi)$ be a symplectic module over a ring $A$ with $2\in A^\times$. A basis of $V$ is a \textit{symplectic basis} if the matrix representing $\phi$ is $J$. 
\end{definition}
For future use, we use certain permutations of a symplectic basis (for example, see the proof of \cref{prop:symptypeofasymprep}).
\begin{lemma}    \label{lem:suitablychangingbasisinflaglemma}
    Let $(V,\phi)$ be a symplectic module of rank $2n$ over a ring $A$ with $2\in A^\times$. Given a symplectic basis $\mathrm{B}$ of $V$, and an integer $1\leq r\leq n$, there exists a basis of $V$, given by permuting elements of $\mathrm{B}$, with respect to which the matrix representing $\phi$ is
    \[
        \begin{pmatrix}
            &&-I_r\\
            &J_{2(n-r)}&\\
            I_r&&
        \end{pmatrix}.
    \]
\end{lemma}
\begin{proof}
    Let $\{w_1,\dots,w_n,v_1,\dots,v_n\}$ be a symplectic basis of $V$ as constructed in \cref{prop:sympbasis}. The required basis is
        \(
            \{w_1,\dots,w_n,v_{r+1},\dots,v_n,v_1,\dots,v_r\}.
        \)
\end{proof}

\subsubsection{Properties of symplectic subspaces} In this section, we recall a few basic facts on the symplectic subspaces of a symplectic vector space $(V,\phi)$ over a field $K$ with $\mathrm{Char}(K)\neq 2$.
\begin{proposition}
\label{prop:symplecticplussymplecticcompletemtiseverything}
    If $W$ is a symplectic subspace of $(V,\phi)$, then 
    \[
        V=W\oplus W^0.
    \]
\end{proposition}

    \begin{proof}
     We choose a symplectic basis $\{w_1,\dots,w_r,v_1,\dots,v_r\}$ of $W$ by \cref{prop:sympbasis}. 
    For $v\in V$, let $\tilde{v}=v+\sum_{i=1}^r(\phi(v,v_i)w_i-\phi(v,w_i)v_i)$. Then  
    \(
        \tilde{v}\in W^0
    \)
    and $v-\tilde{v}\in W$.
    This shows that $W+W^0=V$. The result now follows since $W$ is symplectic and $W\cap W^0=0$.
\end{proof}

    \begin{corollary}\label{cor:symplecticperpissymplectic}
    A subspace $W$ of $(V,\phi)$ is symplectic if and only if $W^0$ is symplectic. Moreover, if $W$ is symplectic, then $(W^0)^0=W$. 
\end{corollary}
\begin{proof}  
    If $W$ is symplectic, then it follows that $W^0$ is symplectic by \cref{prop:symplecticplussymplecticcompletemtiseverything}. 
    For the converse, if $W$ is not symplectic, then there is a nonzero $w\in W\cap W^0$. If $W^0$ were symplectic, then there exists $v\in W^0$ such that $\phi(w,v)=1$. Since $w\in W$ and $v\in W^0$, we have a contradiction. This proves that if $W^0$ is symplectic, then so is $W$. If $W$ and $W^0$ are symplectic, then by looking at dimensions, it follows that $W=(W^0)^0$.
\end{proof}

\begin{corollary}\label{cor:extendingbasisofsymplecticsubspace}
    Let $(V,\phi)$ be a symplectic vector space of dimension $2n$.
    Let $W$ be a symplectic subspace of $V$ of dimension $2r$. Any symplectic basis of $W$ can be extended to a basis of $V$ with respect to which the matrix corresponding to $\phi$ is 
    \[
        \begin{pmatrix}
            J_{2r}&\\&J_{2(n-r)}
        \end{pmatrix}.
    \]
\end{corollary}
\begin{proof}
     We see that $V=W\oplus W^0$ and both $W$ and $W^0$ are symplectic subspaces of $V$ by \cref{prop:symplecticplussymplecticcompletemtiseverything} and \cref{cor:symplecticperpissymplectic}. There exists a symplectic basis $B$ of $W$ and a symplectic basis $B^0$ of $W^0$ by \cref{prop:sympbasis}. The desired basis is $B\sqcup B^0$.
\end{proof}

\subsubsection{Properties of isotropic subspaces} 
Let $(V,\phi)$ be a symplectic vector space of dimension $2n$ over a field $K$ with  $\mathrm{Char}(K)\neq 2$. Let $\Lambda$ be a $\Z_p$-lattice in $V$. The neighbours of the vertex $[\Lambda]$ in the building associated with $\GL(V)$ (see \cref{sec:build}) are in one-to-one correspondence with proper subspaces of $\Lambda/p\Lambda$. When $\Lambda$ is a symplectic lattice, then the neighbours of $[\Lambda]$ in the building associated with $\GL(V,\phi)$ correspond to proper isotropic subspaces of $\Lambda/p\Lambda$.   

In this section, we prove some preliminary results on isotropic subspaces of a symplectic vector space. In particular, given an isotropic flag of $V$, we show that there is a symplectic basis of $V$ that extends a basis of the subspaces in the flag (see \cref{lem:extendingabasiswithrespecttoaflag}). 

\begin{lemma}\label{lem:extendingabasisofisotropictosymplectic}
    Let $W$ be an isotropic subspace of $V$. There is a basis of $W$ that extends to a symplectic basis of $V$. Moreover, 
    the dimension of $W$ is $n$ if and only if $W$ is a Lagrangian subspace of $V$.
\end{lemma}
\begin{proof}

    The proof of the first part follows by induction on $\dim_K W$ for which the base case is \cref{prop:sympbasis}.
    If $W$ is an isotropic subspace of dimension $r$, then $W$ is Lagrangian inside a $2r$-dimensional symplectic subspace of $V$. This proves the second part.
\end{proof}

The following result generalizes \cref{prop:sympbasis} and is crucial. An isotropic flag in a symplectic vector space corresponds to a simplex in the corresponding symplectic building. We use the following \cref{lem:extendingabasiswithrespecttoaflag} in \cref{sec:build} to show that every simplex lies in an apartment (see \cref{prop:simultaneousdiagisotropicsubsapceofsymplectic}). 

\begin{proposition}
    \label{lem:extendingabasiswithrespecttoaflag}
        Let $(V,\phi)$ be a symplectic vector space over a field $K$ with $\mathrm{Char}(K)\neq 2$, and let 
        \[
            0=V_0\subset V_1\subset \dots \subset V_m\subset V
        \]
        be a flag in $V$ such that $V_m$ is isotropic. 
        There is a symplectic basis $\{w_1,\dots,w_n,v_1,\dots,v_n\}$ of $V$ such that 
        $\{w_1,\dots,w_{i_k}\}$ is a basis of $V_{k}$ for $1\leq k\leq m$, where $i_k=\dim V_k$. 
\end{proposition}
    \begin{proof}
        We induct on $m\geq 1$. The base case is \cref{lem:extendingabasisofisotropictosymplectic}. 
        Using the inductive hypothesis on the flag $0\subset V_1\subset \dots \subset V_{m-1}\subset V$, we get a symplectic basis $\{w_1,\dots, w_n,v_1,\dots,v_n\}$ of $V$ such that $\{w_1,\dots,w_{i_k}\}$ is a basis for $V_k$ for $1\leq k< m$.
        For $w\in V_m$, consider the element in $V_m$ defined by $\tilde{w}=w+\sum_{j=1}^{i_{m-1}}\phi(w,v_j)w_j$. By our choice,
        $\phi(\tilde{w},v_j)=0$ for $1\leq j\leq i_{m-1}$.
        Let $\tilde{V}_m=\{\tilde{w} \ |\ w\in V_m\}$. 
        Since $\tilde{V}_m$ is an isotropic subspace of the symplectic space $\{w_1,\dots,w_{i_{m-1}}, v_1,\dots,v_{i_{m-1}} \}^0$, using the induction hypothesis, we obtain a symplectic basis of $\{ w_1,\dots,w_{i_{m-1}}, v_1,\dots,v_{i_{m-1}} \}^0$ extending a basis 
        $\{w_{i_{m-1}+1},\dots,w_{i_m}\}$ of $\tilde{V}_m$. By combining the two bases suitably, the result follows.
    \end{proof}

\subsection{Symplectic representations}
Our version of the generalization of Ribet's lemma is for symplectic representations. 
 In this section, we study the properties of symplectic representations. We make use of the extra structure of symplectic representations to prove \cref{prop:skeqsymgraph} (see \cref{sec:symtype}).

\begin{definition} 
Let $A$ be a ring with $2\in A^\times$.
A  \textit{morphism} between symplectic modules $(V,\phi)$ and $(V',\phi')$ over $A$ is an $A$-module morphism $\varphi\colon V\to V'$ such that, for all $v,w\in V$, 
\[
    \phi'(\varphi(v),\varphi(w))=\lambda(\varphi)\phi(v,w)\quad \textnormal{for some }\lambda(\varphi)\in A.
\] A morphism $\varphi\colon (V,\phi)\to (V',\phi')$ is an \textit{isomorphism} if there exists a two-sided inverse morphism $\varphi^{-1}\colon (V',\phi')\to (V,\phi)$. 
\end{definition}
Note that symplectic morphisms are closed under composition. Moreover, for symplectic morphisms $\varphi\colon V\to V'$ and $\varphi'\colon V'\to V''$, we see that $\lambda(\varphi'\circ \varphi)=\lambda(\varphi')\lambda(\varphi)$. 

    \begin{definition}
        For a symplectic module $(V,\phi)$ over $A$, let $\End_A(V,\phi)$ be the monoid (with respect to composition) of endomorphisms of $(V,\phi)$, and let $\GL_A(V,\phi)$ be the group of units in $\End_A(V,\phi)$ consisting of automorphisms of $(V,\phi)$. 
    The homomorphism $\lambda\colon \GL_A(V,\phi)\to A^\times$ is called the \textit{similitude character} of $(V,\phi)$. We often write $\GL(V,\phi)$ for $\GL_{\Q_p}(V,\phi)$. 
    \end{definition}
    
   \begin{remarks}\quad 
        \begin{enumerate}[(a)]
            \item When $V$ is $2$-dimensional, then  $\End_A(V,\phi)\simeq M_2(A)$ and $\lambda=\det$. 
        
        \item A morphism $\varphi\colon (V,\phi)\to (V',\phi')$ of symplectic $A$-modules is an isomorphism if and only if $\varphi$ is bijective and $\lambda(\varphi)\in A^\times$. 
        \end{enumerate}
   \end{remarks}

\begin{definition}
    Let $(V,\phi)$ be a symplectic module over $A$. 
    A representation $\rho\colon G\to \GL_A(V)$ of a group $G$ is \textit{symplectic} if the image $\rho(G)$ lies in $\GL_A(V,\phi)$. The associated character $\lambda\circ \rho \colon G\to A^\times$ is called the \textit{similitude character} of $\rho$. 
\end{definition}

    \begin{lemma}\label{lem:symrepimpliesperpcongtooriginal}
    Let $(V,\phi)$ be a symplectic $A$-module.
    Let $\rho\colon G\to \GL_A(V)$ be a representation, and let $\lambda$ be a character of $G$.
    Then $\rho$ is symplectic with similitude character $\lambda$ if and only if given a basis of $V$ with respect to which $\phi$ is represented by a matrix $Q$, we have 
    \[
        \rho(g)^t  Q\rho(g)=\lambda(g)Q\quad \textnormal{for all }g\in G.
    \]
    Moreover, if $\rho$ is symplectic with similitude character $\lambda$ and  $\rho^\perp\colonequals \lambda \rho^{-t}$, then $\rho^\perp\simeq \rho$.
\end{lemma}
\begin{proof}
    The proof follows by definition.
\end{proof}

An $A$-linear group representation induces an $A$-algebra representation of the group ring by linearity. 
There is a similar correspondence for symplectic group representations. We make the following definitions to make this correspondence precise.
\begin{definitions}
\quad 
\begin{enumerate}[(i)]
    \item 
    Let $R$ be an $A$-algebra. An \textit{involution} on $R$ is
    a bijective $A$-linear map $\tau\colon R\to R$ such that $\tau(r_1r_2)=\tau(r_2)\tau(r_1)$ for all $r_1,r_2\in R$. Such a tuple $(R,\tau)$ is called an \textit{involutive algebra} over $A$.
    \item  A \textit{morphism} between involutive $A$-algebras $(R,\tau)$ and $(R',\tau')$ is an $A$-algebra map $\psi\colon (R,\tau)\to (R',\tau')$ such that 
    \[
        \psi(\tau(r))=\tau'(\psi(r))\quad \textnormal{for all }r\in R.
    \]
\end{enumerate}
\end{definitions}
Given a skew symmetric matrix $Q\in \GL_{2n}(A)$, we define an involution $\tau_Q$ on $M_{2n}(A)$ by $\tau_Q(M)=Q^{-1}M^tQ$ for all $M\in M_{2n}(A)$. 
\begin{definition}\cite[Definition 2.1]{moakher2023symplecticdeterminantlawsinvariant}
     A \textit{symplectic} representation of an involutive $A$-algebra $(R,\tau)$ is a morphism of involutive algebras $(R,\tau)\to (M_{2n}(A),\tau_Q)$, for a skew symmetric matrix $Q$. By a change of basis, $Q$ can be taken to be $J$.
\end{definition}

Let $(V,\phi)$ be a symplectic module of rank $2n$ over $A$, $\lambda\colon G\to A^\times$ be a character of $G$, and let $\tau=\tau_\lambda$ be the $A$-linear involution $\tau$ on $A[G]$ given by $\tau(g)=\lambda(g)g^{-1}$ for $g\in G$. The following result is an easy consequence of the definitions.
\begin{lemma}
     Let $\rho\colon G\to \GL_A(V)$ be a representation and let $\tilde{\rho}\colon A[G]\to M_{2n}(A)$ be the induced $A$-algebra representation  after fixing a basis of $V$. Let $\lambda\colon G\to A^\times$ be a character. Then $\tilde{\rho}$ is a symplectic representation of $(A[G],\tau_\lambda)$ if and only if $\rho$ is a symplectic representation with similitude character $\lambda$.
\end{lemma}

\subsubsection{The general symplectic group}
Let $(V,\phi)$ be a symplectic module of rank $2n$ over $A$ with a symplectic basis $\mathrm{B}$ (such a basis always exists by \cref{prop:sympbasis}).
For $\sigma\in \GL_A(V,\phi)$,
let $M_\sigma$ be the corresponding matrix in $M_{2n}(A)$ with respect to $\mathrm{B}$.
The map taking $\sigma\in \GL_A(V,\phi)$ to $M_\sigma\in \GL_{2n}(A)$ gives an isomorphism between the group $\GL_A(V,\phi)$ and the $2n$-dimensional \textit{general symplectic group}, denoted by $\GSp_{2n}(A)$, consisting of matrices $M$ in $\GL_{2n}(A)$ satisfying
    \[
    M^t J M=\lambda(M)J\quad\textnormal{ for some } \lambda(M)\in A^\times.
    \]    
After fixing a symplectic basis, a symplectic representation $G\to \GL_A(V,\phi)$ induces a representation $G\to \GSp_{2n}(A)$.

The following result gives a precise picture of how the $n\times n$ blocks of a matrix in the symplectic group relate to each other. The following \cref{lem:GSp2nformulas} follows from the definition of $\GSp_{2n}(A)$ and is used implicitly in our proofs.
    \begin{lemma}\label{lem:GSp2nformulas}
        Let $A$ be a ring with $2\in A^\times$.
        A matrix $M=\begin{pmatrix}
        P_1&P_2\\P_3&P_4
    \end{pmatrix}\in \GL_{2n}(A)$, where $P_i\in M_n(A)$ for $1\leq i\leq 4$,  belongs to $\GSp_{2n}(A)$ if and only if 
    \[
        P_1^tP_3\textnormal{ and }P_2^tP_4 \textnormal{ are symmetric};\textnormal{ and }P_1^tP_4-P_3^tP_2=\lambda(M)I_n.
    \]
    Moreover, if $M\in \GSp_{2n}(A)$ and $P_3^tP_2=0$, then $P_1$ is upper (or lower) triangular if and only if $P_4$ is lower (or upper) triangular.
    \end{lemma}

\section{Symplectic types and constructing idempotents}
    \label{sec:symtype}
The first part of this section is on symplectic representations and their types \cite[Definition 5.3]{moakher2023symplecticdeterminantlawsinvariant}. This sets up notation to prove \cref{prop:skeqsymgraph}.
In the second part, given a residually multiplicity-free symplectic representation $\rho\colon G\to \GSp_{2n}(\Z_p)$, we construct orthogonal idempotents in $\Z_p[G]/\ker \rho$ corresponding to a symplectic basis  (see  \cref{prop:contructingidempotents} and \cref{thm:uppertirangluarwithrespecttog0inGwithdistincteigenvaluesmodp}). 
Later in \cref{sec:build}, we show that the simplicial complex formed by $G$-stable vertices in the symplectic building lies in the apartment with respect to a symplectic basis satisfying the conclusions of \cref{thm:uppertirangluarwithrespecttog0inGwithdistincteigenvaluesmodp} (see \cref{prop:Srhosymliesinanapartment}).

\subsection{Symmetry of symplectic representations} 
When working with a symplectic representation, there is a {symmetry} on the set of its pairwise distinct irreducible factors \cite[\S 1.8]{BellaicheChenevierFamilies}.
The first part of this subsection is to explain this symmetry precisely. Later, we define symplectic types, as in \cite{moakher2023symplecticdeterminantlawsinvariant}, and show that symplectic representations give rise to symplectic types.



\begin{definition}\label{def:psiperp} Let $K$ be a field, and let $(R,\tau)$ be an involutive  $K$-algebra.
    Let $\psi\colon R\to \End_K(V)$ be a representation. Let $V^*\colonequals \Hom_K(V,K)$. We define the representation $\psi^\perp\colon R\to \End_K(V^*)$ as follows: for all $r\in R$, $f\in V^*$, $v\in V$, we define 
\begin{equation}
\label{eq:psi^*definition}
    (rf)(v)=f(\tau(r)v).
\end{equation}
For a representation $\psi\colon R\to M_n(K)$, the representation $\psi^\perp=(\psi\circ \tau)^t$
\end{definition}
\begin{lemma}\label{lem:psiirredupsiperpirred}
    Let $\psi\colon R\to \End_K(V)$ be a representation and let $\psi^\perp\colon R\to \End_K(V^*)$ be the representation on $V^*$ as defined in \cref{def:psiperp}. Then
    \begin{enumerate}[$(a)$]
        \item \label{item:subspacestableiffquotientstable} a subspace $W$ of $V$ is $R$-stable if and only if $(V/W)^*$ is $R$-stable in $V^*$,
        \item \label{item:subsapceofdualisdualofsomequotient} a subspace of $V^*$ is of the form $(V/W)^*$ for a subspace $W$ of $V$.
    \end{enumerate}
    In particular, $\psi^\perp$ is irreducible if and only if $\psi$ is irreducible.
\end{lemma}
\begin{proof} 
We first prove \eqref{item:subspacestableiffquotientstable}.
If $W$ is an $R$-stable subspace of $V$, then $(V/W)^*$ is $R$-stable by 
\cref{eq:psi^*definition}. For the converse, if we let $W'$ be the span of $\{\psi(r)w\ |\ r\in R,w\in W\}$, then for all $f\in (V/W)^*$,  $f$ vanishes on $W'$. Since $W\subseteq W'$ and $(V/W)^*\subseteq (V/W')^*$, using dimensions, we see that $W=W'$. Hence $W$ is $R$-stable. 

We now prove \eqref{item:subsapceofdualisdualofsomequotient}. Given a subspace $\tilde{W}$ of $V^*$, we let $W=\{v\in V\ |\ f(v)=0\textnormal{ for all }f\in \tilde{W}\}$. By definition, we have $\tilde{W}\subseteq (V/W)^*$. 
By another dimension count argument, one can easily show that $\tilde{W}=(V/W)^*$. 
\end{proof}

The following lemma relates the Jordan--H\"older factors of $\psi^\perp$ to those of $\psi$.

\begin{lemma}\label{lem:factorsofpsiperp}
    Let $V$ be a vector space over a field $K$, $(R,\tau)$ be an involutive $K$-algebra, and let $\psi\colon R\to \End_K(V)$ be a representation with $\psi^{\mathrm{ss}}\simeq\psi_1\oplus \dots\oplus \psi_d$, where $\psi_i$ is irreducible for all $1\leq i\leq d$. Then
    \[
        (\psi^\perp)^{\mathrm{ss}}\simeq\psi_1^\perp\oplus \dots\oplus \psi_d^\perp.
    \]
\end{lemma}
\begin{proof}
    When $d=1$, the result is \cref{lem:psiirredupsiperpirred}.
    Assume $d\geq 2$.
    There is a basis of $V$ with respect to which $\psi$ is block upper triangular. That is, there exists $Q\in \GL_{n}(K)$ such that 
    \[
        \psi=Q\begin{pmatrix}
            \psi_{i_1}&*&\dots &*\\
            &\psi_{i_2}&\dots&*\\
            &&\ddots&\vdots\\
            &&&\psi_{i_d}\end{pmatrix}Q^{-1}.
    \]
    For any $r\in R$, we see that
    \[
        \psi^\perp(r)=\psi(\tau(r))^t=Q^{-t}\begin{pmatrix}
            \psi_{i_1}^\perp(r)&& &\\
            *&\psi_{i_2}^\perp(r)&&\\
            \vdots&\vdots&\ddots&\\
            *&*&\dots&\psi_{i_d}^\perp(r)\end{pmatrix} Q^t.
    \]
    Therefore, we get $(\psi^\perp)^{\mathrm{ss}}=\psi_1^\perp \oplus \dots \oplus\psi_d^{\perp}$. 
\end{proof}
Following \cite{BellaicheChenevierFamilies},
we define the \textit{sign} of an irreducible representation $\psi$ such that $\psi\simeq \psi^\perp$. 
\begin{lemma}\cite[Lemma 1.8.7]{BellaicheChenevierFamilies}
    Let $\psi\colon R\to M_{n}(K)$ be an absolutely irreducible representation such that  $\psi^\perp\simeq \psi$. 
    For any $Q\in \GL_n(K)$ such that $\psi^\perp=Q\psi Q^{-1}$, we have $Q^tQ^{-1}=\varepsilon=\pm 1$. The sign $\varepsilon$ does not depend on the choice of $Q$ and is called the \emph{sign of }$\psi$.
\end{lemma}
\begin{proof}
    Since $\psi\simeq (\psi^\perp)^\perp$,
    the proof follows by Schur's lemma. 
\end{proof}
\begin{remark}\label{rem:symplecticiffsign-1}
    Let $(R,\tau)$ be an involutive $K$-algebra, and let $\psi\colon R\to M_{2n}(K)$ be an irreducible representation. It is an easy calculation to see that $\psi$ is symplectic if and only if $\psi\simeq \psi^\perp$ over $K$ with sign $-1$.
\end{remark}



 \begin{proposition}\label{prop:existsinvolutiononI}
    Let $\psi\colon (R,\tau)\to (M_{2n}(K),\tau_Q)$ be a symplectic representation with absolutely irreducible Jordan--H\"older factors, and let $\psi_1,\dots,\psi_d$ be the isomorphism classes of the Jordan--H\"older factors of $\psi$. There exists a involution $\sigma$ on $I\colonequals\{1,2,\dots,d\}$ such that $\psi_{\sigma(i)}\simeq \psi_i^\perp$ for all $i\in I$     
 \end{proposition}
 
\begin{proof}
Let $Q\in \GL_{2n}(K)$ be a skew-symmetric matrix such that $\psi^\perp=Q\psi Q^{-1}$. Since $\psi$ and $\psi^\perp$ are isomorphic, their semisimplifications are isomorphic. The result follows by
\cref{lem:factorsofpsiperp}.
\end{proof}

\subsection{The symplectic extension graph is skew-symmetric}
We are ready to prove  \cref{prop:skeqsymgraph}: the symplectic extension graph associated with an irreducible $p$-adic representation is a skew-symmetric graph. 

Let $(V,\phi)$ be a symplectic $\Q_p$-vector space, $\rho\colon G\to \GL(V,\phi)$ be a $p$-adic representation admitting a $G$-stable primitive lattice (that is, a symplectic $\Z_p$-module of rank $\dim V$), and let $\rho_1,\dots,\rho_d$ be the distinct isomorphism classes of absolutely irreducible residual factors of $\rho$. By \cref{prop:existsinvolutiononI}, there exists a involution $\sigma$ on $\{1,\dots,d\}$ such that $\rho_{\sigma(i)}\simeq \rho_i^\perp$. 

\begin{proof}[Proof of \cref{prop:skeqsymgraph}] For $1\leq i\neq j\leq d$, assume that there is an edge $(\rho_j,\rho_i)$ in $\Gamma_\rho^\GSp$. We need to show that $(\rho_{\sigma(i)},\rho_{\sigma(j)})$ is an edge in $\Gamma_\rho^\GSp$.
    Let $\Lambda$ be a $G$-stable symplectic vertex realizing the edge $(\rho_j,\rho_i)$, and let $\Lambda_1,\Lambda_2,$ and $\Lambda_3$ by $G$-stable lattices such that  
    $p\Lambda\subset \Lambda_1\subset \Lambda_2\subset \Lambda_3\subset \Lambda$ and the subquotient $\Lambda_3/\Lambda_1(=(\Lambda_3/p\Lambda)/(\Lambda_1/p\Lambda))$ is a nonsplit extension $\rho_j$ by $\rho_i$. We have the induced $G$-stable flag on the dual of $\Lambdabar$,
    \[
        0\subset \Hom_{\F_p}(\Lambda/\Lambda_3,\F_p)\subset \Hom_{\F_p}(\Lambda/\Lambda_2,\F_p)\subset \Hom_{\F_p}(\Lambda/\Lambda_1,\F_p)\subset \Hom_{\F_p}(\Lambdabar,\F_p), 
    \]
      where the $G$-action given by $\rhobar_\Lambda^\perp$. If $\Lambda_i'\colonequals \Hom_{\F_p}(\Lambda/\Lambda_i,\F_p)$ for $1\leq i \leq 3$, then the subquotient $\Lambda_1'/\Lambda_3'$ of $\Lambdabar^*$ is a nonsplit extension of $\rho_{\sigma(i)}$ by $\rho_{\sigma(j)}$ as $\Lambda_3/\Lambda_1$ is nonsplit. Since $\rhobar_\Lambda\simeq \rhobar_\Lambda^\perp$ by \cref{lem:symrepimpliesperpcongtooriginal}, we see that $\Lambda$ also realizes the edge $(\rho_{\sigma(i)},\rho_{\sigma(j)})$ in $\Gamma_\rho^\GSp$.
\end{proof}

\subsection{Symplectic group representations}

We define symplectic (GMA) types as in  \cite{moakher2023symplecticdeterminantlawsinvariant}. Let $(V,\phi)$ be a symplectic vector space over a field $K$ with $\mathrm{Char}(K)\neq 2$. Given a symplectic multiplicity-free representation $\psi\colon G\to \GL_K(V,\phi)$, we show that it gives rise to a symplectic type. 
\begin{definition}\label{def:symplectictype}\cite[Definition 5.3]{moakher2023symplecticdeterminantlawsinvariant}
    Let $I=\{1,\dots,d\}$. A \textit{symplectic type} of dimension $2n$ is a data $\delta=((I_0,I_1,I_2),\sigma,(n_i)_{i\in I})$, where 
     \begin{enumerate}[(i)]
    \item $\sigma\colon I\to I$ is a permutation such that $\sigma^2=\mathrm{Id}$,
    \item $I=I_0\sqcup I_1\sqcup I_2$, where $I_0\subset I$ is the subset fixed by $\sigma$, and $\sigma(I_1)=I_2$ such that $\sigma(i)>i$ for all $i\in I_1$,
    \item  \label{item:nievensymtype}
    $n_i$ is even for all $i\in I_0$,
    \item $n_{\sigma(i)}=n_i$ for all $i\in I_1$,
    \item $\sum_{i\in I} n_i=2n$.
\end{enumerate}

\end{definition}
\begin{example} When $d=2$, we have two symplectic types. Let $\delta=((I_0,I_1,I_2),\sigma,(n_1,n_2))$ be a symplectic type of dimension $2n$. If $\sigma$ is trivial, then $I_0=\{1,2\}$ and $n_1,n_2$ are even by \cref{def:symplectictype}\eqref{item:nievensymtype}. If $\sigma$ is nontrivial, then $I_1=\{1\}$ and $n_1=n_2=n$. 
\end{example}
Let $(V,\phi)$ be a symplectic vector space over a field $K$ with $\mathrm{Char}(K)\neq 2$, and let $\psi\colon G\to \GL_K(V,\phi)$ be a multiplicity free  symplectic representation of a group $G$ with similitude character $\lambda$. Let $\psi^{\mathrm{ss}}\simeq \psi_1\oplus\dots\oplus \psi_d$ and let $\psi_i$ be absolutely irreducible for $1\leq i\leq d$. Let $R=K[G]$ and $\tau\colon R\to R$ be the $K$-linear involution such that $\tau(g)=\lambda(g)g^{-1}$ for all $g\in G$. By \cref{prop:existsinvolutiononI}, there exists a involution $\sigma$ on $I=\{1,\dots,d\}$ such that $\psi_{\sigma(i)}\simeq\psi_i^\perp$. 
Let $I_0=\{i\in I\ |\ \sigma(i)=i\}$, $I_1=\{i\in I\ |\ i<\sigma(i)\}$, and let $d_1=\#I_1$. Let $n_i$ be the dimension of $\psi_i$ for all $i\in I$.

\begin{proposition}\label{prop:symptypeofasymprep}
    The data $\delta=((I_0,I_1,\sigma(I_1)),\sigma,(n_i)_{i\in I})$ defines a symplectic type of dimension $2n$.
\end{proposition}
When $\psi$ is the residual representation associated with an irreducible symplectic representation over a local domain with residue field $K$,
\cref{prop:symptypeofasymprep} follows from \cite[Lemma 1.8.8]{BellaicheChenevierFamilies}. Moreover, \cref{prop:symptypeofasymprep} is a special case of \cite[Theorem 5.12]{moakher2023symplecticdeterminantlawsinvariant}. Since we are working over $\Z_p$, we give a simpler proof of \cref{prop:symptypeofasymprep}. Our proof also works for complete discrete valuation rings.
The only thing left to prove is that $n_i$ is even for all $i\in I_0$, which we prove by showing that $\rho_i$ is symplectic for all $i\in I_0$.
The rest of this section is devoted to proving \cref{prop:symptypeofasymprep}.

\begin{proposition}
    \label{lem:irreducibleGstableisisotropicorsymplectic} Let $W$ be an irreducible $G$-stable subspace of $(V,\phi)$ on which the $G$-action is given by $\psi_i$ for some $i\in I$. Then $W$ is either isotropic or symplectic. Moreover, $W$ is isotropic if and only if $i\not\in I_0$. 
\end{proposition}
\begin{proof}
    Let $W'=\{v\in W\ |\ \phi(v,W)=0\}=W\cap W^0$. Since $W$ is $G$-stable, $W'$ is $G$-stable. 
    Moreover, $W$ is irreducible as a $G$-module implies that $W'\in \{0,W\}$. 
    This shows that $W$ is either an isotropic or a symplectic subspace of $(V,\phi)$.

    Suppose that $W$ is isotropic. By \cref{lem:suitablychangingbasisinflaglemma}, there exists a basis of $V$ extending a basis of $W$ such that 
    \[
        \psi=\begin{pmatrix}
            \psi_i&*&*\\
            &P_1&P_2\\
            &P_3&P_4
        \end{pmatrix}\quad \textnormal{and}\quad \begin{pmatrix}
                &&-I_{n_i}\\&J_{2(n-n_i)}&\\I_{n_i}&&
            \end{pmatrix}\psi \begin{pmatrix}
                &&-I_{n_i}\\&J_{2(n-n_i)}&\\I_{n_i}&&
            \end{pmatrix}^{-1}=\psi^\perp,
    \]
        where $P_4$ is a square matrix of dimension $n_i$, $P_1$ is a square matrix of dimension $2(n-n_i)$, and other block matrices have suitable dimensions.
            From the latter equation, we see that $P_3\equiv 0$ on $G$, $P_1$ and $P_4$ are representations of $G$,  and $P_4=\psi_i^\perp$. Since $\psi$ is multiplicity-free, we see that $i\not\in I_0$.
            If $W$ is symplectic, then the result follows by \cref{prop:psikispsimperp} below.
\end{proof}

\begin{proposition}
    \label{prop:psikispsimperp}
        Let 
        $0=V_0\subset V_1\subset \dots\subset  V_j\subset V$
        be a $G$-stable flag such that $V_{j-1}$ is isotropic, and $V_{i}/V_{i-1}$ is irreducible with the $G$-action given by $\psi_i$ for $i\in \{1,\dots,j\}$. 
        Let $1\leq k\leq j-1$ be the least integer such that $\phi(V_{k}\times V_{j})\neq 0$. Then $\sigma(k)=j$.
\end{proposition}
\begin{proof}
    Let $V_k'=\{w\in V_k\ |\ \phi(w,V_j)=0\}=V_k\cap V_{j}^0$, which is $G$-stable. We see that $V_k'$ contains $V_{k-1}$ and $V_k\neq V_k'$  by hypothesis.
    Since $V_k/V_{k-1}$ is irreducible as a $G$-module, we get $V_k'=V_{k-1}$.
    Similarly, $V_j'=V_j\cap V_k^0$ is $G$-stable and $V_j'=V_{j-1}$.
    Thus, we have shown that the induced pairing 
    \[
        \phi\colon (V_k/V_{k-1})\times(V_j/V_{j-1})\to K 
    \] is perfect. Therefore, $\phi$ induces an isomorphism between $V_k/V_{k-1}$ and $(V_j/V_{j-1})^*$ as $G$-modules, where the $G$-action on $(V_j/V_{j-1})^*$ is given  by $\psi_j^\perp$.
    This shows that $\psi_k\simeq \psi_j^\perp$.
\end{proof}

Firstly, we prove \cref{prop:symptypeofasymprep} under the assumption that $I_0=I$.

\begin{lemma}\label{lem:I0=Isemsimple}
    Assume $I_0=I$. Then $\psi$ is semisimple and $n_i$ is even for all $i\in I_0$. Moreover, there is a basis of $V$ with respect to which $\psi=\varphi_0$.
\end{lemma}
\begin{proof}
    We induct on $d$. The result holds for $d=1$ since $\psi$ is irreducible and symplectic. Assume $d\geq 2$. Let $\psi_k$ be a subrepresentation of $\psi$. Since $k\in I_0$, there is a $G$-stable symplectic subspace of $V$ on which the $G$-action is given by $\psi_k$ by \cref{lem:irreducibleGstableisisotropicorsymplectic}. This implies $n_i$ is even. By \cref{cor:extendingbasisofsymplecticsubspace}, there is a basis of $V$ such that 
    \[
        \psi=\begin{pmatrix}
            \psi_i&*\\&\varphi
        \end{pmatrix}\quad \textnormal{and}\quad \begin{pmatrix}
            J_{n_i}&\\&J_{2n-n_i}
        \end{pmatrix}\psi\begin{pmatrix}
            J_{n_i}&\\&J_{2n-n_i}
        \end{pmatrix}^{-1}=\psi^\perp.
    \]
    The latter equation implies that $*=0$ and $J\varphi J^{-1}=\varphi^\perp$. By induction, $\varphi$ is semisimple, and $n_i$ is even for $i\in I_0-\{k\}$. This proves the result. 
\end{proof}

The following result shows that a maximal  $G$-stable isotropic flag has length $d_1$ in $V$. This is the dimension of the simplicial complex of $G$-stable lattices in the symplectic building (see \cref{prop:dimofSrhosym}). 

\begin{proposition}\label{prop:maximalGstableisotropicflag}
    Assume $I_0\neq I$. There exists a maximal $G$-stable isotropic flag 
    \[
        0=V_0\subset V_1\subset \dots \subset V_{d_1}\subset V,
    \]
    where the $G$-action on $V_k/V_{k-1}$ is given by $\psi_{i_k}$ for all $1\leq k\leq d_1$ and $\psi_{i_{k_1}}\not\simeq \psi_{i_{k_2}}^\perp$ for all $1\leq k_1,k_2\leq d_1$.
\end{proposition}
\begin{proof}
    We first claim that $V$ has an isotropic $G$-stable subspace. Let $W$ be an irreducible $G$-stable subspace of $V$. If $W$ is isotropic, then let $V_1=W$. If not, then $W$ is symplectic by \cref{lem:irreducibleGstableisisotropicorsymplectic}. By the proof of \cref{lem:I0=Isemsimple}, if the $G$-action on $W$ is given by $\psi_{i_0}$ for $i_0\in I_0$, then there exists a subrepresentation $\varphi$ of $\psi$ such that \[
    \psi^{\mathrm{ss}}=\varphi^{\mathrm{ss}}+\psi_{i_0}.
    \]
    The claim now follows by induction on $d$ by the assumption $I_0\neq I$.

    Let \[
        0=V_0\subset V_1\subset \dots \subset V_{j}\subset V,
    \] be a maximal $G$-stable isotropic flag
    where the $G$-action on $V_k/V_{k-1}$ is given by $\psi_{i_k}$, with $\psi_{i_{k_1}}\not\simeq \psi_{i_{k_2}}^\perp$, for all $1\leq k_1,k_2\leq j$. Let $m=\sum_{k=1}^jn_{i_k}$.
    By \cref{lem:suitablychangingbasisinflaglemma} for the above isotropic flag, there exists a basis of $V$ such that 
    \[
        \psi=\begin{pmatrix}
            \varphi&*&*\\
            &P_1&P_2\\
            &P_3&P_4
        \end{pmatrix}\quad \textnormal{and}\quad \begin{pmatrix}
        &&-I_m\\
        &J&\\
        I_m&&
    \end{pmatrix}\psi \begin{pmatrix}
        &&-I_m\\
        &J&\\
        I_m&&
    \end{pmatrix}^{-1}=\psi^\perp.
    \] 
    Since $\varphi$ is invertible, the latter equation implies that $P_3\equiv 0$. Moreover, we see that $JP_1J^{-1}=P_1^\perp$ and $P_4=\varphi^\perp$. If $j<d_1$, then there is a subrepresentation $\psi_{i_{j+1}}$ of $P_1$ such that $i_{j+1}\notin I_0$ and $\psi_{i_{j+1}}\not\simeq \psi_k^\perp$ for any $1\leq k\leq j$. That is, there exists a $G$-stable subspace $V_{j+1}$ of $V$, containing $V_j$, such that $G$-acts on $V_{j+1}/V_j$ by $\psi_{i_{j+1}}$. Since $i_{j+1}\notin I_0$, the subspace $V_{j+1}$ is isotropic in $V$ by \cref{prop:psikispsimperp}. This contradicts the maximality of $j$. Thus, $j=d_1$.
\end{proof}

    \begin{proof}[Proof of \cref{prop:symptypeofasymprep}] 
\cref{prop:maximalGstableisotropicflag} guarantees the existence of a maximal $G$-stable isotropic flag in $V$. Let $m=\sum_{i\in I_1} n_i$, and let $\{w_1,\dots,w_n,v_{m+1},\dots,v_n,v_1,\dots,v_m\}$ be a basis of $V$ as found in \cref{lem:suitablychangingbasisinflaglemma}. With respect to such a basis, 
\[
    \psi=\begin{pmatrix}
        P&*&*\\&\varphi&*\\&&P^\perp
    \end{pmatrix}\quad \textnormal{and}\quad \begin{pmatrix}
        &&-I_m\\
        &J&\\
        I_m&&
    \end{pmatrix}\psi \begin{pmatrix}
        &&-I_m\\
        &J&\\
        I_m&&
    \end{pmatrix}^{-1}=\psi^\perp.
\]
The latter equation implies $J\varphi J^{-1}=\varphi^\perp$. Since $\varphi$ is symplectic with Jordan--H\"older factors of $\{\psi_i\}_{i\in I_0}$,  $\varphi$ is semisimple and $n_i$ is even for all $i\in I_0$ by \cref{lem:I0=Isemsimple}. 
\end{proof}

\subsection{Constructing idempotents} 

Let $(\Lambda_0,\phi)$ be a symplectic $\Z_p$-module of rank $2n$, and let $\rho\colon G\to \GL_{\Z_p}(\Lambda_0,\phi)$ be a symplectic representation that is residually multiplicity free with absolutely irreducible residual factors $\rho_1,\dots,\rho_d$. 
We make the following assumption in this section.
\begin{equation}
    \label{assum:distincteigenvaluesmodp}\tag{dist}
    \textnormal{There exists }  g_0\in G \textnormal{ such that } \rhobar(g_0)\textnormal{ has pairwise distinct eigenvalues}.
\end{equation}



In this section, we construct certain idempotents in $\Z_p[G]/\ker \rho$ corresponding to a symplectic basis of $\Lambda_0$. This is used later to show that every $G$-stable vertex in the symplectic building lies in an apartment (see \cref{prop:Srhosymliesinanapartment}). We construct our idempotents in the spirit of Urban \cite[Proof of Theorem 1.1]{UrbanKlingen}.

We introduce some notation to state the result below. Let $m_0=\sum_{i\in I_0}n_i$. Fix a diagonal embedding $\iota\colon \oplus_{i\in I_0}M_{n_i}(\F_p)\hookrightarrow M_{m_0}(\F_p)$. Let $\varphi_0$ be the composition of $\oplus_{i\in I_0}\rho_i$ with $\iota$, and let $J_0=\iota((J_{n_i})_{i\in I_0})$.

\begin{theorem}\label{thm:uppertirangluarwithrespecttog0inGwithdistincteigenvaluesmodp}
    Let $(\Lambda_0,\phi)$ be a symplectic $\Z_p$-module of rank $2n$.
    Let $\rho\colon G\to \GL_{\Z_p}(\Lambda_0,\phi)$ be a symplectic residually multiplicity-free representation with absolutely irreducible residual factors $\rho_1,\dots,\rho_d$. Assume that $\rho$ satisfies \eqref{assum:distincteigenvaluesmodp}. 
    Then there exists a basis of $\Lambda_0$ with respect to which 
    \begin{enumerate}[(a)]
        \item $\rho(g_0)$ is diagonal,
        \item for $m_1=\sum_{i\in I_1}n_i$,  \[
    \bar{\rho}=\begin{pmatrix}
        P&*&*\\
        &\varphi_0&*\\
        &&P^\perp
    \end{pmatrix}\quad \textnormal{and}\quad 
    \begin{pmatrix}
        0&0&-I_{m_1}\\
        0&J_0&0\\
        I_{m_1}&0&0
    \end{pmatrix}\bar{\rho} \begin{pmatrix}
        0&0&-I_{m_1}\\
        0&J_0&0\\
        I_{m_1}&0&0
\end{pmatrix}^{-1}=\bar{\rho}^\perp,
\]
where 
\[
    P=\begin{pmatrix}
        \rho_{i_1}&*&*\\
        &\ddots&*\\
        &&\rho_{i_{d_1}}
    \end{pmatrix} \quad \textnormal{and}\quad i_k\not\in I_0.
\]
    \end{enumerate}
\end{theorem}

\begin{remark} When $n=1$, the assumption \eqref{assum:distincteigenvaluesmodp} is equivalent to $\rho$ being multiplicity free using trace-determinant relations \cite[2.1.1]{Bellaicheimages}.
\end{remark}


We start with the following general lemma.
\begin{lemma}\label{lem:simultaneousdiagwrtaflag}
         Let $(\Lambda_0,\phi)$ be a symplectic $\Z_p$-module of rank $2n$, $M\in \GL_{\Z_p}(\Lambda_0,\phi)$ with eigenvalues $\{\alpha_1,\dots,\alpha_{2n}\}$ that are distinct modulo $p$, and let $p\Lambda_0\subset {\Lambda}_1\subset \dots \subset {\Lambda}_n\subset {\Lambda}_0$ be a flag such that $\Lambda_i/p\Lambda_0$ is isotropic and  stable under $\bar{M}$ for all $1\leq i\leq n$. Then there exists a symplectic basis $\{w_1,\dots,w_n,v_1,\dots,v_n\}$ of $\Lambda_0$ such that 
        \begin{enumerate}[$(a)$]
            \item $\Lambda_i/p\Lambda_0=\langle \bar{w}_1,\dots,\bar{w}_{i} \rangle$ for $1\leq i\leq n$,
            \item $M=\mathrm{diag}(\alpha_{i_1},\dots,\alpha_{i_{2n}})$.
        \end{enumerate}
    \end{lemma}
    \begin{proof}
        Let us denote by $\bar{V}_i$ the subspace $\Lambda_i/p\Lambda_0$ of $\Lambdabar_0$ for all $1\leq i\leq n$. 
        We induct on $n\geq 1$. Let $n=1$. If $\bar{V_1}=\langle \bar{w}\rangle$, then $\bar{M}\bar{w}=\bar{\alpha}_i\bar{w}$ for some $1\leq i\leq 2$. Let $w_1\in \Lambda_0$ such that $Mw_1=\alpha_iw_1$. Since $\bar{\alpha}_1\neq \bar{\alpha}_2$, $\bar{w_1}=a\bar{w}$ for some $a\in \F_p^\times$. Dividing by $a$, we can assume that $\bar{w}_1=\bar{w}$. For $v_1\in V$ such that $Mv_1=\alpha_jv_1$ and $j\neq i$, since $\phi$ is nondegenerate on $\Lambda_0$, we necessarily have $\phi(v_1,w_1)\neq 0\bmod{p}$. By replacing $v_1$ by $\phi(v_1,w_1)^{-1}v_1$, we get a symplectic basis $\{w_1,v_1\}$ with respect to which the result holds. 

        Assume the result holds for $m<n$.
        Let $0\subset\bar{V}_1\subset\dots \subset \bar{V}_n\subset \bar{\Lambda}_0$ be an $\bar{M}$-stable isotropic flag. Let $w_1\in V$ such that $\bar{V_1}=\langle\bar{w}_1\rangle$  and $Mw_1=\alpha_{i_1} w_1$ (the argument is the same as in the previous paragraph). Let $\{w_1,w_2',\dots,w'_n,v'_1,\dots,v'_n\}$ be a symplectic basis of $\Lambdabar_0$, extending a symplectic basis of $\Lambda_0$ such that $\bar{V}_i=\langle \bar{w}_1,\bar{w'_2},\dots,\bar{w'_i}\rangle$ for $1<i\leq n$. With respect to the basis $\{w_1,\dots,w_n',v_2',\dots,v_n',v_1'\}$, 
        \[
            M=\begin{pmatrix}
            \alpha_{i_1}&*_1&*_2\\
            &M'&*_3\\
            &*_4&\alpha_{i_k}
        \end{pmatrix}\quad \textnormal{and}\quad M^t\begin{pmatrix}
            &&-1\\&J&\\1&&
        \end{pmatrix}M=\lambda(M)\begin{pmatrix}
            &&-1\\&J&\\1&&
        \end{pmatrix}.
        \]
        From the latter equation, $*_4=0$ and $\alpha_{i_1}\alpha_{i_k}=\lambda(M)$. The aim is to choose an eigenvector candidate for $v_1'$ to ensure that,  for all $i\in \{1,2,3\}$,
        $*_i$ is trivial. Let $v_1\in \Lambda_0$ be such that $Mv_1=\alpha_{i_k} v_1$. Let $v_1=a_1w_1+\sum_{i=2}^n a_iw'_i + \sum_{i=1}^n b_i v'_i$. We claim that $b_1=\phi(v_1,w_1)$ and that it is a unit in $\Z_p$. By our choice of $v_1$,  $\bar{v}_1$ has eigenvalue $\bar{\alpha}_{i_k}$. If $b_1\equiv 0\bmod p$, then $\bar{v}_1$ lies in a stable subspace of $\Lambdabar_0$ on which the action is given by $\begin{pmatrix}
            \alpha_{i_1}&*_1\\
            &M'
        \end{pmatrix}$. This implies $\bar{\alpha}_{i_k}=\bar{\alpha}_j$ for some $j\neq i_k$, which is a contradiction. Hence, $b_1\in \Z_p^\times$. Dividing by $b_1$, we can assume $b_1=\phi(v_1,w_1)=1$. 
        Let $W_1=\langle w_1,v_1\rangle$, and let $W_2=\Lambda_0\cap W_1^0$. Since $W_1$ is $M$-stable, we see that $W_2$ is $M$-stable. Indeed, since $MW_1=W_1$, we have 
        \[
            \phi(MW_2,W_1)=\phi(MW_2,MM^{-1}W_1)= \lambda(M)\phi(W_2,W_1)=0.
        \]
        Moreover, $W_2$ is seen to be a symplectic $\Z_p$-module of rank $2(n-1)$. Let $\tilde{\Lambda}_i=\{w-\phi(v_1,w)w_1 \ |\ w\in \Lambda_i\}$. Using the induction hypothesis on $M|_{W_2}$ with respect to the flag $pW_2\subset \tilde{\Lambda}_2\subset\dots \subset \tilde{\Lambda}_n\subset W_2$, we get a symplectic basis $\{w_2,\dots,w_n,v_2,\dots,v_n\}$ of $W_2$ with respect to which $\tilde{\Lambda}_i/pW_0=\{\bar{w}_2,\dots,\bar{w}_i\}$ and $M|_{W_2}$ is diagonal. Therefore, the required basis is $\{w_1,\dots,w_n,v_1,\dots,v_n\}$.
    \end{proof}

    \begin{proof}[Proof of \cref{thm:uppertirangluarwithrespecttog0inGwithdistincteigenvaluesmodp}] By Hensel's lemma, $\rho(g_0)$ has eigenvalues in $\Z_p$, that are distinct modulo $p$.
    Let $\Lambda_1,\dots,\Lambda_{d_1}$ be $\Z_p$-submodules of $\Lambda_0$ such that  $p\Lambda_0\subset \Lambda_1\subset \dots \subset\Lambda_{d_1}\subset \Lambda_0$, and $\Lambda_{k}/p\Lambda_0$ is $G$-stable and isotropic in $\Lambdabar_0$ for all $1\leq k\leq d_1$.  Let $W_0=0$, $W_k=\Lambda_k/p\Lambda_0$, and let the $G$-action on $W_{k}/W_{k-1}$ be given by $\rho_{i_k}$ for all $1\leq k\leq d_1$. Moreover, we can extend the given $G$-stable isotropic flag to a $G$-stable flag
    \[
        0=W_0\subset W_1\subset \dots \subset W_{d_1}\subset W_{d_1+1}\subset \dots\subset W_{d_1+d_0}\subset \Lambdabar_0, 
    \]
    where the $G$-action on $W_{d_1+k}/W_{d_1+k-1}$ is given by $\rho_{j_k}$ for every $k\in \{1,\dots,d_0\}$ and $j_k\in I_0$.
    
    Since $0\subset W_1\subset \dots \subset W_{d_1+d_0}\subset \Lambdabar_0$ is also $\bar{\rho}(g_0)$-stable, we can consider a refined isotropic flag $0=V_0\subset V_1\subset \dots \subset V_{n}\subset \Lambdabar_0$ such that $W_k=V_{\dim W_k}$ for all  $1\leq k\leq d_1$.  There is a symplectic basis $\{w_1,\dots,w_n,v_1,\dots,v_n\}$ of $\Lambda_0$ with respect to which $\rho(g_0)$ is diagonal with eigenvalues distinct modulo $p$ and $V_i=\langle \bar{w}_1,\dots,\bar{w}_{i} \rangle $ for all $1\leq i\leq n$ by \cref{lem:simultaneousdiagwrtaflag}.  If $m=\dim W_k$, then the required basis is $\{w_1,\dots, w_n,v_{m+1},\dots v_n,v_1,\dots,v_{m}\}$.
    \end{proof}
    Using the idempotents we construct in \cref{prop:contructingidempotents}, we remark that \cref{thm:uppertirangluarwithrespecttog0inGwithdistincteigenvaluesmodp} is a version of \cite[Proposition 2.4.2 (iii)]{Bellaicheimages} over coefficient rings that are complete discrete valuation rings with finite residue field. 
\begin{corollary}
    \label{prop:contructingidempotents}
Assume there exists a basis of $\Lambda_0$ satisfying the conclusions of \cref{thm:uppertirangluarwithrespecttog0inGwithdistincteigenvaluesmodp}. There exist orthogonal idempotents $\{e_i\}_{i=1}^{2n}$ in $\Z_p[G]/\ker\rho$, and a fixed-point-free involution $\eta$ on $\{1,\dots,2n\}$ such that 
        \begin{enumerate}[$(a)$]
            \item \label{item:idempotents1} $e_1+\dots+e_{2n}=1$,
            \item \label{item:idempotents2} For $i\neq \eta(j)$, $(e_i+e_j)\Lambda_0$ is isotropic in $\Lambda_0$,
            \item \label{item:idempotents3} For all $i\in \{1,\dots, 2n\}$,  $e_i\Lambda_0+e_{\eta(i)}\Lambda_0$ is a symplectic submodule of $\Lambda_0$. 
        \end{enumerate}
\end{corollary}
    
    Similar to \cite[Proof of Theorem 1.1]{UrbanKlingen}, we need the following classical result to prove \cref{prop:contructingidempotents}.
    \begin{lemma}\label{lem:Mgeneratesdiagonalmatrices}
        Let $M=\mathrm{diag}(\alpha_1,\dots,\alpha_{2n})\in \GL_{2n}(\Z_p)$ such that $\alpha_i\not\equiv \alpha_j\bmod{p}$ for all $1\leq i\neq j\leq 2n$. The subalgebra of $M_{2n}(\Z_p)$ generated by $M$ is the diagonal matrices. 
    \end{lemma}

    \begin{proof}[Proof of \cref{prop:contructingidempotents}] Fix a basis of $\Lambda_0$ satisfying the conclusions of \cref{thm:uppertirangluarwithrespecttog0inGwithdistincteigenvaluesmodp}.  
        By renaming the $\rho_i$'s, we can assume that \cref{thm:uppertirangluarwithrespecttog0inGwithdistincteigenvaluesmodp} holds with $i_j=j$ for $1\leq j\leq d_1$.
    For $1\leq i\leq 2n$, there exists $e_i\in \Z_p[G]/\ker\rho$ such that $\rho(e_i)=\mathrm{diag}(0,\dots,0,1,0,\dots,0)$, where $1$ is in the $i$-th position, by \cref{lem:Mgeneratesdiagonalmatrices}. The fixed-point-free involution $\eta$ on $\{1,\dots,2n\}$ is defined such that 
    \[
        \rho(e_{\eta(i)})=\rho^\perp(e_i).
    \]
    By our choice of idempotents and
    involution $\eta$, the result follows.
    \end{proof}

\section{Buildings}
\label{sec:build}
A \textit{Bruhat--Tits building} is a simplicial complex that is a union of subcomplexes called \textit{apartments}, which are Coxeter complexes, and satisfies certain axioms \cite[Definition 4.1]{AbramenkoBrown}. For connected semisimple linear algebraic groups over $\Q_p$, the buildings constructed correspond to infinite \textit{Weyl groups}, giving us {Euclidean buildings} that are $\mathrm{CAT}(0)$ metric spaces \cite[Chapter 11]{AbramenkoBrown}.
In this section, we define the (Euclidean) Bruhat--Tits building associated with $\GL_n(\Q_p)$ and $\GSp_{2n}(\Q_p)$ whose vertices correspond to homothety classes of certain $\Z_p$-lattices. 

Let $(V,\phi)$ be a symplectic $\Q_p$-vector space, and let $\rho\colon G\to \GL(V,\phi)$ be a symplectic representation that stabilizes a primitive lattice with residual factors $\rho_1,\dots,\rho_d$ that are pairwise nonisomorphic and absolutely irreducible. The simplicial complex $\scrS_\rho^{\GL}$ formed by $G$-stable vertices in the building associated with $\GL(V)$ is compact and bounded if and only if $\rho$ is irreducible, and has dimension $d-1$ \cite[Proposition 3.2.1 and Proposition 3.3.2]{BellaicheRibetslemmaPadova}. 

We study the subcomplex $\scrS_\rho^\GSp$ of $\scrS_\rho^\GL$ formed by vertices in the building associated with $\GL(V,\phi)$, find its dimension (see  \cref{prop:dimofSrhosym}), and, assuming $\rho$ satisfies \eqref{assume:dist}, we show that $\scrS_\rho^\GSp$ lies in a sub-apartment of the building associated with $\GL(V,\phi)$ (see \cref{prop:Srhosymliesinanapartment}). 

The condition $\sigma(i)\neq i$ for all $1\leq i\leq d$ is equivalent to the existence of neighbouring symplectic vertices in $\scrS_\rho^\GSp$ (see \cref{thm:symplecticneighboursiffI0empty}).  Under these equivalent assumptions, if $\rho$ is irreducible, then we can move between $G$-stable symplectic vertices and prove the existence of such vertices on the boundary (see \cref{prop:existssymplecticwithnoneighbourswithspecificJH}). The existence of symplectic vertices in the boundary of $\scrS_\rho^\GSp$ is equivalent to the existence of certain edges in $\Gamma_\rho^\GSp$ (see \cref{prop:equivnonexistneighexitstext}).
This is used in the proof of \cref{thm:ourmain}.

\subsection{Building associated with general linear groups}
Let $V$ be a $\Q_p$-vector space and $\Lambda$ be a lattice in $V$.
The homothety class of $\Lambda$ is $\{\alpha\Lambda\ |\ \alpha\in \Q_p^\times\}$ and is denoted by $[\Lambda]$.
We define Bruhat--Tits buildings as a simplicial complex with homothety classes of lattices as its vertices. This is not the most natural definition, but a useful one to visualize and work with. For more equivalent definitions, see Abramenko--Brown \cite[Chapters 4, 5, 6]{AbramenkoBrown}.

\begin{definition} \label{def:GLnbuilding}
Let $V$ be an $n$-dimensional $\Q_p$-vector space.
The \textit{Bruhat--Tits building} associated with $\GL(V)$ is the following simplicial complex denoted by $\mathcal{X}_n$.
    \begin{enumerate}
        \item A vertex (0-simplex) is a homothety class of a lattice in $V$.
\item An edge (1-simplex) between vertices $x_1,x_2$ of $\mathcal{X}_n$ exists if there are representatives $\Lambda_1$ and $\Lambda_2$ for $x_1$ and $x_2$ respectively, such that 
\[
	p\Lambda_1\subset \Lambda_2\subset \Lambda_1.
\]
Such an $x_2$ is called a \textit{neighbour} of $x_1$ in $\mathcal{X}_n$.
\item An $m$-simplex in $\mathcal{X}_n$ is given by $(m+1)$-vertices $x_i=[\Lambda_i]$ for all $0\leq i\leq m$, satisfying 
	\[
		p\Lambda_0\subset \Lambda_m \subset \Lambda_{m-1} \subset \dots \subset\Lambda_1\subset \Lambda_0.
	\]
    \end{enumerate}
\end{definition}

\begin{remark}\label{rem:GLnbuildingneighbours}
\quad 
    \begin{enumerate}[(a)]
        \item The neighbour relation is symmetric: If $x'$ is a neighbour of $x$, then $x$ is also a neighbour of $x'$. To see this, choose representative lattices $\Lambda$ and $\Lambda'$ of $x$ and $x'$ respectively, such that 
        \[
            p\Lambda\subset \Lambda'\subset \Lambda.
        \]
        This implies that $p(p^{-1}\Lambda')\subset \Lambda\subset p^{-1} \Lambda'$. Thus choosing the new representative $p^{-1}\Lambda'$ for $x'$, we see that $x$ is a neighbour of $x'$. 
        \item \label{item:neighboursandpropersubspaces} The neighbours of a vertex $[\Lambda]$ in $\mathcal{X}_n$ are in one-to-one correspondence with the proper $\F_p$-subspaces of $\Lambda/p\Lambda$. For a neighbour $[\Lambda']$ of $[\Lambda]$ such that $p\Lambda\subset \Lambda'\subset \Lambda$ the associated proper $\F_p$-subspace is $\Lambda'/p\Lambda$. 
    \end{enumerate}
\end{remark}

 \begin{example}
Let $V$ be a $\Q_p$-vector space of dimension $2$. The building $\mathcal{X}_2$ is a $(p+1)$-regular infinite tree. Let $\rho\colon G\to \GL(V)$ be an irreducible representation that admits a $G$-stable lattice. Suppose that $\rho$ is residually reducible and is given by two distinct characters $\chi_1$ and $\chi_2$. The subcomplex $\scrS_\rho^\GL$ of $\mathcal{X}_2$ formed by homothety classes of $G$-stable lattices is a finite tree \cite[Proposition 3.2.1]{BellaicheRibetslemmaPadova}. The two end vertices of $\scrS_\rho^\GL$ realize the two edges in the graph $\Gamma_\rho^\GL$.    
 \end{example}

\subsection{Building associated with general symplectic groups}
We define the Bruhat--Tits building associated with symplectic groups analogous to \cref{def:GLnbuilding} (also see \cite[Chapter 20]{PaulGarrett}). 
We start by defining symplectic and primitive lattices in symplectic $\Q_p$-vector spaces. The vertices in the symplectic building are either  homothety classes of primitive lattices, or neighbours of such primitive vertices in $\mathcal{X}_{2n}$ that correspond, under \cref{rem:GLnbuildingneighbours}\eqref{item:neighboursandpropersubspaces}, to an isotropic subspace.

Let $(V,\phi)$ be a symplectic vector space over $\Q_p$ of dimension $2n$. 
Given a basis $\{e_1,\dots,e_{2n}\}$ of $V$, we write $\langle e_1,\dots,e_{2n} \rangle $ for its $\Z_p$-span.  
Let $\Lambda$ be a lattice of $V$. There exists an integer $k\in \Z$ such that $\phi(\Lambda\times \Lambda)=p^k\Z_p$. Indeed, if $\Lambda=\langle e_1,\dots,e_{2n} \rangle$, then $k$ is the minimum of the $p$-adic valuations of $\phi(e_i,e_j)$ for $1\leq i,j\leq 2n$. Let $\bar{\Lambda}=\Lambda/p\Lambda$, and for $v\in \Lambda$, let $\bar{v}$ denote the element $v+{p\Lambda}$ in $\Lambdabar$. Define the bilinear form  $\bar{\phi}_\Lambda\colon \Lambdabar\times\Lambdabar\to \F_p$ by
\(
    \bar{\phi}_\Lambda(\bar{v},\bar{w})=\phi(v,w)p^{-k}\bmod{p},\ \textnormal{for all } v,w\in \Lambda.
\)
Since \(
\phi(p\Lambda\times \Lambda)p^{-k}=\phi(\Lambda\times p\Lambda)p^{-k}=p^{k+1}p^{-k}\Z_p=p\Z_p,\)
the map $\bar{\phi}_\Lambda$ is well defined. Moreover, $\bar{\phi}_\Lambda$  is  alternating as $\phi$ is alternating. 
\begin{definition}\label{def:symplattice}
    A lattice $\Lambda$ of $V$ is \textit{symplectic} if $\bar{\phi}_\Lambda$ is a symplectic pairing on $\bar{\Lambda}$. A symplectic lattice $\Lambda$ is  \textit{primitive} if $\phi(\Lambda\times\Lambda)=\Z_p$.
\end{definition}
As remarked in \cref{sec:background}, in this section we prove that if a $\Z_p$-lattice $\Lambda$  in a symplectic $\Q_p$-vector space $(V,\phi)$ is primitive, then $(\Lambda,\phi|_{\Lambda\times \Lambda})$ is a symplectic $\Z_p$-module of rank $\dim V$ (see \cref{cor:primlatticeissympmod}). 

\begin{proposition}\label{prop:symlatticescharacterizationviasymbasis}
    Let $\{w_1,\dots,w_n,v_1,\dots,v_n\}$ be a symplectic basis of $V$. The lattice 
    \[
        \Lambda=
        \langle p^{a_1}w_1,\dots,p^{a_n}w_n,p^{b_1}v_1,\dots,p^{b_n}v_n\rangle
    \] is symplectic if and only if $a_i+b_i=a_j+b_j$ for all $1\leq i,j\leq n$. Moreover, $\Lambda$ is primitive if and only if $a_i+b_i=0$ for all $1\leq i\leq n$.
\end{proposition}

\begin{proof}
    Let $k=\min\{a_i+b_i\ |\ 1\leq i\leq n\}$. Since $\{w_1,\dots,w_n,v_1,\dots,v_n\}$ is a symplectic basis, we observe that $\phi(\Lambda\times \Lambda)=p^k\Z_p$. Moreover, for any $1\leq i\leq n$, the element $w_i$ pairs trivially with all the elements of the basis except $v_i$. If $a_i+b_i>k$ for some $i$, then $p^{a_i}w_i$ is such that 
    \[
        \phi(p^{a_i}w_i,p^{b_i}v_i)\equiv 0\bmod{p^{k+1}}.
    \] 
    This implies that the alternating form $\bar{\phi}_\Lambda$ on $\Lambdabar$ is degenerate. Hence, if $\Lambda$ is symplectic, then $a_i+b_i=k$ for all $1\leq i\leq n$.

    For the converse, let $a_i+b_i=k$ for all $1\leq i\leq n$. Then
    \[  \phi(p^{a_i}w_i,p^{b_i}v_j)=p^k\delta_{ij},\quad \textnormal{for all }1\leq i,j\leq n.
    \]
    Since $\phi(w_i,w_j)=\phi(v_i,v_j)=0$ and $\phi(w_i,v_j)=\delta_{ij}$ for all $1\leq i,j\leq n$, it follows that $\bar{\phi}_\Lambda$ is given by $J$ with respect to the basis $\{p^{a_1}w_1,\dots,p^{a_n}w_n,p^{b_1}v_1,\dots,p^{b_n}v_n\}$. This shows that $\Lambda$ is symplectic. Since $\phi(\Lambda\times \Lambda)=p^k\Z_p$, we see that $\Lambda$ is primitive if and only if $k=a_i+b_i=0$ for all $1\leq i \leq n$.
\end{proof}

\begin{definition}
Let $(V,\phi)$ be a $2n$-dimensional symplectic $\Q_p$-vector space.
The \textit{Bruhat--Tits building} associated with $\GL_{\Q_p}(V,\phi)$ is the simplicial complex denoted by $\Delta_n(V,\phi)$ defined by the following properties.
\begin{enumerate}
    \item A vertex (0-simplex) is either a homothety class of a primitive lattice, or a homothety class with a representative $\Lambda$ such that there is a primitive lattice $\Lambda_0$ such that 
\[	p\Lambda_0\subset \Lambda\subset \Lambda_0
\] and the subspace $\Lambda/p\Lambda_0$ is an isotropic subspace of $\bar{\Lambda}_0$. 
\item An edge (1-simplex) between vertices  $x_1,x_2$ in $\Delta_n$ exists if there are representatives $\Lambda_1$ and $\Lambda_2$ for $x_1$ and $x_2$ respectively, and a primitive lattice $\Lambda_0$ such that 
\[
	p\Lambda_0\subseteq \Lambda_1, \Lambda_2\subseteq
	\Lambda_0,
\]
$\Lambda_i/p\Lambda$ is isotropic for $i\in \{1,2\}$, and either $\Lambda_1\subset \Lambda_2$ or $\Lambda_2\subset \Lambda_1$.
Such a pair $x_1,x_2$ is said to be \textit{neighbours} in $\Delta_n$.
\item An $m$-simplex in $\Delta_n$ is given by $(m+1)$-vertices $x_i=[\Lambda_i]$ for $0\leq i\leq m$ where $\Lambda_0$ is a primitive lattice, and 
	\[
		p\Lambda_0\subset \Lambda_{1} \subset \Lambda_{2} \subset \dots \subset\Lambda_{m}\subset \Lambda_0,
	\]
    where $\Lambda_{m}/p\Lambda_0$ is isotropic. We denote such an $m$-simplex by $[x_0,\dots,x_{m}]$. A maximal simplex is called a \textit{chamber}.
    \item The \textit{dimension} of an $m$-simplex is $m$.
    The \textit{dimension} of $\Delta_n$ is the dimension of a chamber inside it.
\end{enumerate} 
\end{definition}

We note that a vertex in $\Delta_n(V,\phi)$ is by definition a primitive lattice or a neighbour of a primitive lattice.
For convenience, sometimes we write $\Delta_n$ for $\Delta_n(V,\phi)$. 
\begin{lemma}
    The dimension of $\Delta_n$ is $n$.
\end{lemma}
\begin{proof}
    Let $\Lambda_0$ be a primitive lattice.
     Since a chamber containing $[\Lambda_0]$ in $\Delta_n$ corresponds to a maximal isotropic flag in $\Lambdabar_0$, the dimension of a chamber in $\Delta_n$ is the dimension of a Lagrangian subspace of $\Lambdabar_0$. Such a Lagrangian subspace has dimension $n$ by \cref{lem:extendingabasisofisotropictosymplectic}.
\end{proof}

If $\Lambda$ is a symplectic lattice, then
the symplectic $\F_p$-vector space $(\Lambdabar,\bar{\phi}_\Lambda)$ only depends on the homothety class $[\Lambda]$.
We note that not all vertices in $\Delta_n$ correspond to symplectic lattices (see \cref{exp:1ballinanaprtment}). This motivates the following definition.
\begin{definition}
    A \textit{symplectic (primitive) vertex} in $\Delta_n$ is the homothety class of a symplectic (primitive) lattice.
\end{definition}
 It is also common to call a primitive vertex a \textit{hyperspecial vertex}, and a nonprimitive symplectic vertex a \textit{special vertex}. 
\begin{example}\label{exp:1ballinanaprtment}
Let $\{w_1,w_2,v_1,v_2\}$ be a symplectic basis of $\Q_p^4$. For integers $a_1,a_2,b_1,b_2$, let $[a_1,a_2,b_1,b_2]$ denote the homothety class $[\langle p^{a_1}w_1 , p^{a_2} w_2, p^{b_1}v_1,p^{b_2}v_2\rangle]$. The shape of the $1$-ball around $[0,0,0,0]$ in $\Delta_2$ is seen in \cref{fig:1ballinsymplecticgraph}. Each dashed line represents an edge between symplectic vertices. 
\end{example}
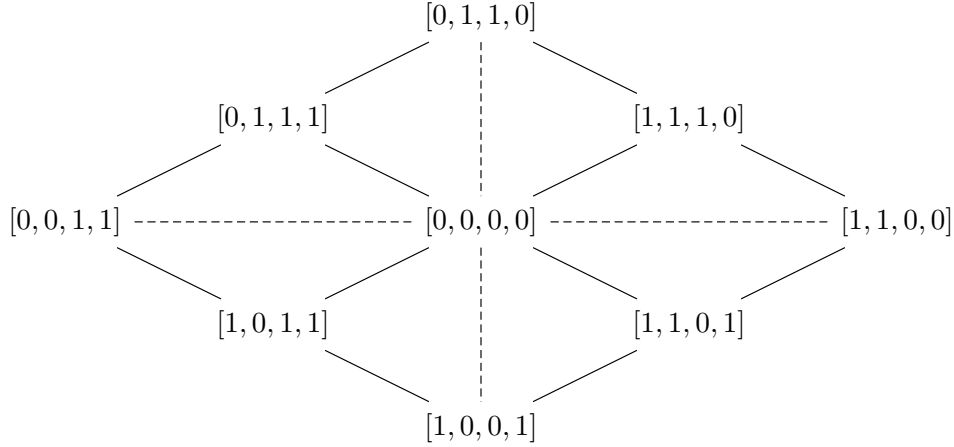
\begin{figure}
    \small
    \centering
    \[
	\begin{tikzcd}
                                                            &                                                                         & {[0,1,1,0]} \arrow[dd, no head, dashed] \arrow[rd, no head]                                                                     &                                                     &             \\
                                                            & {[0,1,1,1]} \arrow[rd, no head] \arrow[ru, no head] \arrow[ld, no head] &                                                                                                                                 & {[1,1,1,0]} \arrow[rd, no head]                     &             \\
{[0,0,1,1]} \arrow[rd, no head] \arrow[rr, no head, dashed] &                                                                         & {[0,0,0,0]} \arrow[ld, no head] \arrow[rd, no head] \arrow[ru, no head] \arrow[dd, no head, dashed] \arrow[rr, no head, dashed] &                                                     & {[1,1,0,0]} \\
                                                            & {[1,0,1,1]} \arrow[rd, no head]                                         &                                                                                                                                 & {[1,1,0,1]} \arrow[ld, no head] \arrow[ru, no head] &             \\
                                                            &                                                                         & {[1,0,0,1]}                                                                                                                     &                                                     &            
\end{tikzcd}
	\]
    \caption{1-ball around $[0,0,0,0]$.}
    \label{fig:1ballinsymplecticgraph}
\end{figure}

\subsubsection{Symplectic vertices and their neighbours}
Let $(V,\phi)$ be a symplectic vector space over $\Q_p$ of dimension $2n$.
Let $x$ be a symplectic vertex in $\Delta_n$ and let $y$ be its neighbour in $\Delta_n$. By definition, there exists representative lattices $\Lambda$ and $\Lambda'$ of $x$ and $y$ respectively, and a primitive lattice $\Lambda_0$ such that $p\Lambda_0\subset \Lambda',\Lambda\subset \Lambda_0$ with $\Lambda'\subset \Lambda$ or $\Lambda\subset \Lambda'$. We show in this section that $\Lambda/p\Lambda_0$ is a Lagrangian subspace of $\Lambda_0/p\Lambda_0$ (see  \cref{cor:symplecticneighboursofprimitive}). Therefore, we must have $p\Lambda \subset \Lambda'\subset \Lambda$. Moreover, one can show that $\Lambda'/p\Lambda$ is isotropic in $\Lambda/p\Lambda$ (for example, using \cref{prop:simultaneousdiagisotropicsubsapceofsymplectic}).


We start with some basic properties of primitive lattices in a symplectic vector space. Using this, we show that primitive lattices are symplectic $\Z_p$-modules (see \cref{cor:primlatticeissympmod}).
\begin{proposition}
\label{prop:primitiveissymplecticplussymplectic}
    Let $\Lambda$ be a primitive lattice of $V$. Let $w_1,v_1\in \Lambda$ such that $\phi(v_1,w_1)=1$, $\Lambda_1=\langle w_1,v_1 \rangle$ be the $\Z_p$-span of $\{w_1,v_1\}$, and let $\Lambda_2=\Lambda_1^0\cap \Lambda$. Then
    \(
        \Lambda=\Lambda_1\oplus \Lambda_2.
    \)
    Moreover, $\Lambda_2$ is a primitive lattice of $\Lambda_2\otimes \Q_p$ with respect to the symplectic form induced by $\phi$.
\end{proposition}
\begin{proof}
        Since $\phi(\Lambda\times \Lambda)=\Z_p$, for any $w\in \Lambda$, the element $\tilde{w}=w+\phi(w,v_1)w_1-\phi(w,w_1)v_1$ lies in $\Lambda_2$. Hence, $\Lambda=\Lambda_1+\Lambda_2$. Moreover, $\Lambda_2\cap \Lambda_1=0$ since $\Lambda_1$ is a symplectic submodule of $\Lambda$. This shows that $\Lambda=\Lambda_1\oplus \Lambda_2$. 
        
        We claim that the symplectic vector space $\Lambdabar$ decomposes as \[\Lambdabar=(\Lambda_1+p\Lambda)/p\Lambda\oplus (\Lambda_2+p\Lambda)/p\Lambda.\] 
        Indeed, if $v=v_1+v_2$, for $v_1\in \Lambda_1$ and $v_2\in \Lambda_2$, then $\bar{v}=\bar{v}_1+\bar{v}_2=(\bar{v}_1-p\bar{w})+(\bar{v}_2+p\bar{w})$. This proves that the two spaces span $\Lambdabar$. To prove the above direct sum decomposition, we show that $(\Lambda_1+p\Lambda)\cap (\Lambda_2+p\Lambda)=p\Lambda$. Let $v_1\in \Lambda_1$ and $v_2\in \Lambda_2$ such that $v_1+pv=v_2+pw$ lies in the intersection $(\Lambda_1+p\Lambda)\cap (\Lambda_2+p\Lambda)$. Since $\Lambda=\Lambda_1\oplus\Lambda_2$, there exists $w_1\in \Lambda_1$ and $w_2\in \Lambda_2$ such that $v_1-v_2=p(w-v)=pw_1-pw_2$. Since $\Lambda_1\cap \Lambda_2=0$, we see that $v_i=pw_i$ for $i\in \{1,2\}$. This proves the claim.

        We show that $\Lambda_2$ is a primitive lattice.
        Since $\phi(v_1,w_1)=1$, we see that $(\Lambda_1+p\Lambda)/p\Lambda$ is a symplectic subspace of $\Lambdabar$. We see that
    $(\Lambda_2+p\Lambda)/p\Lambda$ is also symplectic by \cref{cor:symplecticperpissymplectic}. If $\phi(\Lambda_2\times \Lambda_2)\subseteq  p\Z_p$, then $\bar{\phi}_\Lambda\equiv 0$ on $(\Lambda_2+p\Lambda)/p\Lambda$. This proves that $\phi(\Lambda_2\times \Lambda_2)=\Z_p$. We only need to show $\Lambdabar_2=\Lambda_2/p\Lambda_2$ is a symplectic vector space. Let $w\in \Lambda_2$ such that $\phi(w,\Lambda_2)\subseteq p\Z_p$.
    Since $\Lambda_2+p\Lambda=\Lambda_2+p\Lambda_1$ and $\phi(\Lambda_2\times \Lambda_1)=0$, we have $\phi(w,\Lambda_2+p\Lambda)\subseteq p\Z_p$. Thus, the element $w\in p\Lambda$ as $\Lambdabar$ is symplectic. Let $w_2\in \Lambda_2$ such that $w-pw_2\in p\Lambda_1$. Since $w,w_2\in \Lambda_2$ and $\Lambda_2\cap \Lambda_1=0$, we must have $w=pw_2\in p\Lambda_2$. This proves the result.  
\end{proof}

\begin{lemma}\label{lem:liftingsymplecticbasismodp}
Let $(V,\phi)$ be a $2n$-dimensional symplectic $\Q_p$-vector space and let $\Lambda_0$ be a primitive lattice in $V$. A symplectic basis of $\Lambdabar_0$ can be lifted to a symplectic basis of $\Lambda_0$. 
\end{lemma}
\begin{proof}
    We induct on $n$. By \cref{prop:sympbasis}, the residual symplectic vector space $\Lambdabar_0$ has a symplectic basis; say
    $\{\bar{w}_1,\dots,\bar{w}_n,\bar{v}_1,\dots,\bar{v}_n\}$.
    Let $w_1$ and $v_1$ be a lift of $\bar{w}_1$ and $\bar{v}_1$, respectively, in $ \Lambda_0$. Since $\phi(v_1,w_1)\equiv 1\bmod{p}$, by replacing $v_1$ by $\phi(v_1,w_1)^{-1}v_1$, we can assume that $\phi(v_1,w_1)=1$.  Let $\Lambda_1=\langle w_1,v_1\rangle$, and let $\Lambda_2=\Lambda_0\cap \Lambda_1^0$. 
    The lattice $\Lambda_2$ is primitive in $\Lambda_2\otimes \Q_p$ by \cref{prop:primitiveissymplecticplussymplectic}. Moreover, for a lift $w\in \Lambda_0$ of $\bar{w}\in \{\bar{w_i},\bar{v_i}\}_{i=2}^n$, we see that $\tilde{w}$ lies in $\Lambda_2$ and lifts $\bar{w}$, as $\phi(w,v_1)\equiv \phi(w,w_1)\equiv 0\bmod{p}$.
    Therefore, there exist lifts of elements of $\{\bar{w_i},\bar{v_i}\}_{i=2}^n$ in $\Lambda_2$.
    Since the dimension of $\Lambda_2\otimes \Q_p$ is $2(n-1)$, by the induction hypothesis, there is a symplectic basis $\{w_2,\dots,w_n,v_2,\dots,v_n\}$ of $\Lambda_2$ lifting the symplectic basis $\{\bar{w}_2,\dots,\bar{w}_n,\bar{v}_2,\dots,\bar{v}_n\}$. Therefore, the symplectic basis $\{w_1,\dots,w_n,v_1,\dots,v_n\}$ for $\Lambda_0$ lifts our given symplectic basis of $\Lambdabar_0$. 
\end{proof}
The following corollary is immediate by \cref{lem:liftingsymplecticbasismodp}.
\begin{corollary}\label{cor:primlatticeissympmod}
       A $\Z_p$-lattice in $(V,\phi)$ is primitive if and only if it is spanned by a symplectic basis. In particular, if  $\Lambda$ is primitive, then $(\Lambda,\phi_{\Lambda\times \Lambda})$ is a symplectic $\Z_p$-module. 
\end{corollary}
The following result ensures that every $m$-simplex in $\Delta_n$ is contained in an apartment (see \cref{def:apartment}) of $\Delta_n$. 
\begin{proposition}\label{prop:simultaneousdiagisotropicsubsapceofsymplectic}
    Let $\Lambda_0$ be a primitive lattice in $V$ and let $\Lambda_{1},\dots,\Lambda_{m}$ be lattices such that $p\Lambda_0\subset \Lambda_1\subset \dots\subset \Lambda_m\subset \Lambda_0$ and the subspace $\Lambda_m/p\Lambda_0$ is isotropic in $\Lambdabar_0$. For all $1\leq i\leq m$, let $k_i=\dim_{\F_p}(\Lambda_i/p\Lambda_0)$. There exists a symplectic basis $\{w_1,\dots,w_n,v_1,\dots,v_n\}$ of $\Lambda_0$ such that 
    \[
        \Lambda_i=\langle w_1,\dots,w_{k_i},pw_{k_i+1},\dots,pw_n,pv_1,\dots,pv_n \rangle.
    \]
\end{proposition}
\begin{proof}
    Let $V_i=\Lambda_i/p\Lambda_0$ for $1\leq i\leq m$. By \cref{lem:extendingabasiswithrespecttoaflag}, there is a symplectic basis  $\{\bar{w}_1,\dots,\bar{w}_n,\bar{v}_1,\dots,\bar{v}_n\}$ of $\Lambdabar_0$ such that $\{\bar{w_1},\dots,\bar{w_{k_i}}\}$ spans $V_i$ for all $1\leq i\leq m$. There is a symplectic basis $\{w_1,\dots,w_n,v_1,\dots,v_n\}$ of $\Lambda_0$ such that $w_i$ lifts $\bar{w}_i$ for all $1\leq i\leq n$ by \cref{lem:liftingsymplecticbasismodp}. For all $1\leq i\leq m$, let \[
        V_i'=\langle w_1,\dots,w_{k_i},pw_{k_{i+1}},\dots,pw_n,pv_1,\dots pv_n\rangle/p\Lambda_0.
    \]
    Since $\dim_{\F_p} V_i'=k_i$ and $V_i'\subseteq V_i$, the result follows.
\end{proof}

We can now characterize symplectic neighbours of a primitive vertex $[\Lambda_0]$ in $\Delta_n$.  


\begin{corollary}\label{cor:symplecticneighboursofprimitive}
    Let $\Lambda_0$ be a primitive lattice in $V$ and let $\Lambda$ be a lattice such that $p\Lambda_0\subset \Lambda\subset\Lambda_0$ and  $\Lambda/p\Lambda_0$ is an isotropic subspace of $\Lambdabar_0$. The lattice $\Lambda$ is symplectic if and only if $\Lambda/p\Lambda_0$ is Lagrangian.
\end{corollary}
\begin{proof}
    By \cref{prop:simultaneousdiagisotropicsubsapceofsymplectic}, there exists a symplectic basis $\{w_1,\dots,w_n,v_1,\dots,v_n\}$ of $\Lambda_0$ such that \[\Lambda=\langle w_1,\dots,w_s,pw_{s+1},\dots,pw_n,pv_1,\dots,pv_n\rangle.\]
    It follows that $\Lambda$ is symplectic if and only if $s=n$ by \cref{prop:symlatticescharacterizationviasymbasis}. 
\end{proof}

\begin{corollary}\label{cor:chambercontainsprimitivenon-primitive}
    A chamber in $\Delta_n$ contains exactly one primitive vertex and exactly one special vertex.
\end{corollary}
\begin{proof}
    By definition, a chamber in $\Delta_n$ contains a primitive vertex $x_0$. 
    Let $[x_0,\dots,x_n]$ be a chamber. Choose representative lattices $\Lambda_i$ for $x_i$ for all $i\in \{0,\dots,n\}$ such that $\Lambda_0$ is primitive, and 
     $p\Lambda_0\subset \Lambda_1\subset \dots \subset \Lambda_n\subset \Lambda_0$. This implies that $\Lambda_n/p\Lambda_0$ is Lagrangian in $\Lambdabar_0$. The lattice $\Lambda_n$ must be symplectic with $\phi(\Lambda_n\times \Lambda_n)=p\Z_p$ by \cref{cor:symplecticneighboursofprimitive}.
\end{proof}

\begin{corollary}\label{cor:symneighoneofthemisprimitive}
    Let  $x,x'$ be neighbouring symplectic vertices in $\Delta_n$. Then one of $x$ or $x'$ is represented by a primitive lattice. 
\end{corollary}

\begin{proof}
    Let $\Lambda$ and $\Lambda'$ be representatives of $x$ and $x'$ respectively, and $\Lambda_0$ be a primitive lattice such that $p\Lambda_0\subseteq \Lambda'\subset\Lambda \subseteq\Lambda_0$ and $\Lambda/p\Lambda_0$ is isotropic. Then by  \cref{cor:symplecticneighboursofprimitive}, either  $\Lambda'=p\Lambda_0$ or $\Lambda=\Lambda_0$.
\end{proof}

\begin{proposition}\label{prop:symplecticlatticeisneighbourtoaprimitive}
    Let $x$ be a special vertex in $\Delta_n$. There exists a primitive vertex $x_0\in \Delta_n$ such that $x$ and $x_0$ are neighbours in $\Delta_n$.
\end{proposition}
\begin{proof}
    Since $x$ is a special vertex, we can find a representative lattice $\Lambda$ for $x$ such that $\phi(\Lambda\times \Lambda)=p\Z_p$. Since $x$ is symplectic, there is a $\Z_p$-basis  $\{w_1,\dots,w_n,v_1,\dots,v_n\}$ of $\Lambda$ that lifts a symplectic basis of $\Lambdabar$; that is, we have $\phi(v_i,w_i)=p$ for all $1\leq i\leq n$, $\phi(v_i,w_j)\equiv 0\bmod{p^2}$ for all $1\leq i\neq j\leq n$, and $\phi(v_i,v_j)\equiv \phi(w_i,w_j)\equiv 0\bmod{p^2}$ for all $1\leq i, j\leq n$. Let $\Lambda_0$ be the $\Z_p$-span of $\{w_1,\dots,w_n,p^{-1}v_1,\dots, p^{-1}v_n\}$. We see that $p\Lambda_0\subset \Lambda\subset \Lambda_0$ and $\phi(\Lambda_0\times \Lambda_0)=\Z_p$ since $\phi(w_1,p^{-1}v_1)=1$. We only need to show $\Lambdabar_0$ is symplectic. Let $w_0\in \Lambda_0$ such that $\phi(w_0,w)\equiv 0\bmod{p}$ for all $w\in \Lambda_0$. If $w_0=\sum_{i=1}^n a_i w_i +\sum_{i=1}^n b_i p^{-1}v_i$, then 
    \[
        a_i\equiv \phi(w_0,p^{-1}v_i)\equiv 0\bmod{p} \quad \textnormal{and}\quad b_i\equiv \phi(w_0,w_i)=0\bmod{p},
    \]
    for all $1\leq i\leq n$.
    This shows that $w_0\in p\Lambda_0$. 
\end{proof}
We can characterize symplectic neighbours of symplectic vertices in $\Delta_n$ by \cref{prop:symplecticlatticeisneighbourtoaprimitive,cor:symplecticneighboursofprimitive}.
\begin{corollary}\label{cor:symplecticneighboursofsymplecticlattices}
    Let $x=[\Lambda]$ be a symplectic vertex such that $\phi(\Lambda\times \Lambda)\in \{\Z_p,p\Z_p\}$. Let $y=[\Lambda']$ be a neighbour of $x$ in $\Delta_n$ such that $p\Lambda\subset \Lambda'\subset \Lambda$. Then $y$ is symplectic if and only if $\Lambda'/p\Lambda$ is Lagrangian in $\Lambdabar$.
\end{corollary}
\begin{proof}
    If $\Lambda$ is primitive, then the result is the same as \cref{cor:symplecticneighboursofprimitive}. Let $\Lambda$ be symplectic with $\phi(\Lambda\times\Lambda)=p\Z_p$.
    Let $\mathrm{B}$ be a symplectic basis of $\Lambdabar$ extending a basis of $\Lambda'/p\Lambda$.
    We can lift $\mathrm
    B$ to a basis $\{w_1,\dots,w_n,v_1,\dots,v_n\}$ of $\Lambda$. By the same argument as in the proof of \cref{prop:symplecticlatticeisneighbourtoaprimitive}, the lattice $\Lambda_0=\langle p^{-1}w_1,\dots,p^{-1}w_n,v_1,\dots,v_n\rangle$ is primitive. Moreover, we see that $p\Lambda\subset \Lambda'\subseteq p\Lambda_0\subset\Lambda\subset \Lambda_0$. Since $p\Lambda_0/p\Lambda$ is spanned by the images of $w_1,\dots,w_n$ in $\Lambdabar$, it is Lagrangian in $\Lambdabar$. Since $p\Lambda_0\subset p^{-1}\Lambda'\subseteq\Lambda_0$, the vertex $y$ is symplectic if and only if $\Lambda'=p\Lambda_0$ by \cref{cor:symplecticneighboursofprimitive}. Thus, $y$ is symplectic if and only if $\Lambda'/p\Lambda$ is Lagrangian in $\Lambdabar$.
\end{proof}

\subsubsection{Apartments and sub-apartments in the symplectic building} A building is a union of apartments, which are much simpler to work with. Crucial to our proof of the main theorem is showing that every $G$-stable vertex in $\Delta_n$ lies in a single apartment (see \cref{prop:Srhosymliesinanapartment}).

Let $(V,\phi)$ be a symplectic $\Q_p$-vector space of dimension $2n$. Using \cite[Lemma 1.3.5]{AlisonBuildings}, we define apartments in the building $\Delta_n$.
\begin{definition}\label{def:apartment}
    An \textit{apartment} in $\Delta_n=\Delta_n(V,\phi)$ corresponding to a symplectic basis $\mathrm{B}=\{w_1,\dots,w_n,v_1,\dots,v_n\}$ of $(V,\phi)$ is the simplicial complex $\Sigma=\Sigma(\mathrm{B})$ formed by vertices in $\Delta_n$ of the form $$ [\langle p^{a_1}w_1,\dots,p^{a_n}w_n,p^{b_1}v_1,\dots,p^{b_n}v_n\rangle],$$ where  $a_i,b_i\in \Z$ for all $1\leq i\leq n$. We denote such a vertex in $\Sigma$ by $[a_1,\dots,a_n,b_1,\dots,b_n]$.
\end{definition}

\begin{remark}    
Let $\Sigma$ be an apartment in $\Delta_n$ corresponding to a symplectic basis of $V$. Let $x_0=[0,\dots,0]$ be the origin of $\Sigma$. For all $0\leq i\leq n$, we let $x_i=[0,\dots,0,1\dots,1]$ where the number of zeros is $i$. We call the chamber $[x_0,x_1,\dots,x_n]$ around $x_0$ as the \textit{fundamental chamber} in $\Sigma$. Translating the fundamental chamber by the action of the affine Weyl group 
of $\GSp_{2n}(\Q_p)$, we obtain $\Sigma$. For an example, see \cite[Example 2.3]{Shemanskesymplecticbuilding}. 
\end{remark}

Let $\Sigma$ be an apartment of $\mathcal{X}_{2n}$ with respect to a symplectic basis of $V$. Not all vertices of the form $[a_1,\dots,a_n,b_1,\dots,b_n]\in \Sigma$ correspond to vertices in $\Delta_n$. For example, when $n=2$, the vertex $[0,1,0,1]$ is not a vertex in $\Delta_2$.
By \cref{prop:symlatticescharacterizationviasymbasis}, we know the characterization of symplectic vertices in $\Sigma$.  
The following result gives a criterion for when a vertex $[a_1,\dots,a_n,b_1,\dots,b_n]$ belongs to $\Delta_n$.

\begin{lemma} \cite[Proposition 1.3.4]{AlisonBuildings}\label{lem:charverticesinapartment}
    Let $\Sigma$ be an apartment in $\mathcal{X}_{2n}$ corresponding to a symplectic basis of $V$.
    A vertex $y=[a_1,\dots,a_n,b_1,\dots,b_n]\in \Delta_n$ if and only if there exists an $m\in \Z$ such that $a_i+b_i+2m\in \{1,2\}$ for all $1\leq i\leq n$.
\end{lemma}

We define sub-apartments of $\Delta_n$ analogous to \cite[\S 2.2]{BellaicheRibetslemmaPadova}.
Let $\mathcal{P}=\{B_j\}_{j=1}^d$ be a partition of $\{1,\dots,2n\}$, and let \[\Z^{2n}_\mathcal{P}=\{(a_1,\dots,a_{2n})\in \Z^{2n}\ |\ a_i=a_{i'} \textnormal{ for any }i,i'\in B_j\}.\] We define a bijective mapping 
\begin{equation}\label{eq:subaptmap}
    \Z^{2n}_\mathcal{P}\to \Z^d,
    \quad (a_1,\dots,a_n)\mapsto (b_1,\dots,b_d),
\end{equation}
where $a_i=b_j$ for all $i\in B_j$.

\begin{definition}
    A \textit{sub-apartment} of an apartment $\Sigma$ in  $\Delta_n$ is a subcomplex of $\Sigma$ that, under \cref{eq:subaptmap}, corresponds to an apartment in $\Delta_{m}$ for some $m\leq n$. The dimension of the sub-apartment in this case is $m$. 

    We say a subcomplex $\Sigma'$ of $\Sigma$ is \textit{contained in an apartment} of $\Delta_m$ if $\Sigma'$ is contained in a sub-apartment of dimension $m$. 
\end{definition}

Let $\mathrm{B}$ be a fixed symplectic basis of $V$, and let $\Sigma$ be the corresponding apartment in $\Delta_n$. Let $\Sigma_\mathcal{P}$ be the subcomplex of $\Sigma$ whose vertices have representative in $\Z_\mathcal{P}^{2n}$. 
Under further assumptions on $\mathcal{P}$, we show that $\Sigma_\mathcal{P}$ is a sub-apartment of $\Sigma$ in $\Delta_n$.

\begin{lemma}\label{lem:insidesymplecticpartmentd0d1}
    Let $\mathrm{B}=\{w_1,\dots,w_n,v_1,\dots,v_n\}$ be a symplectic basis of $V$, and let \[
        (\epsilon_1,\dots,\epsilon_{2n})\colonequals (w_1,\dots,w_n,v_1,\dots,v_n).
    \]
    Let $I=\{1,\dots,d\}$, and let $\mathcal{P}=\{B_j\}_{j\in I}$ be a partition of $\{1,\dots,2n\}$ such that
    \begin{enumerate}
        \item there is an order $2$ permutation $\sigma$ on $I$;
        \item for all $k\neq \sigma(j)$, the subspace spanned by $\{\epsilon_i\}_{i\in B_j\cup B_k}$ is isotropic;
        \item for all $j\in  I$, the subspace spanned by $\{\epsilon_i\}_{i\in B_{j}\cup B_{\sigma(j)}}$ is symplectic. 
    \end{enumerate}
    Let $I_0=\{j\in I\ |\ j=\sigma(j)\}$, $I_1=\{j\in I\ |\ j<\sigma(j)\}$, and let $d_s=\#I_s$ for all $s\in \{0,1\}$. Then the subcomplex $\Sigma_\mathcal{P}$ is a sub-apartment of the $\Sigma$ of dimension $d_0+d_1$.
\end{lemma}
\begin{proof} 
Since $B_j\cup B_{\sigma(j)}$ is symplectic, $i\in B_j$ if and only if $n+i\in B_{\sigma(j)}$. For each $j\in I_0\sqcup I_1$, there exists $1\leq i_j\leq n$ such that $\{\epsilon_{i_j},\epsilon_{n+i_{j}}\}$ spans a symplectic subspace of $V$. 
Let $V'$ be the symplectic subspace of $V$ generated by $\mathrm{B}'=\{\epsilon_{i_1},\dots, \epsilon_{i_{d_0+d_1}},\epsilon_{n+i_1},\dots, \epsilon_{n+i_{d_0+d_1}}\}$.
    By definition, $\Sigma_\mathcal{P}$ is the subcomplex of $\Sigma$ whose vertices are of the form $[a_1,\dots,a_{2n}]$ with $a_i=b_j$ for all $i\in B_j$. Let us define the map 
    \[
        \Z_\mathcal{P}^{2n}\to \Z^{2(d_0+d_1)},\quad [a_1,\dots,a_{2n}]\mapsto[b_1,\dots,b_{2(d_0+d_1)}].
    \]
    We see that
    $b_j+b_{\sigma(j)}=a_i+a_{n+i}$ for any $i\in B_j$. We see that $[b_1,\dots,b_{2(d_0+d_1)}]\in \Delta_{d_0+d_1}$ if and only if $[a_1,\dots,a_{2n}]\in \Delta_n$ by \cref{lem:charverticesinapartment}.
    This shows that $\Sigma_\mathcal{P}$ is a sub-apartment of $\Sigma$ of dimension $d_0+d_1$.
\end{proof}

 \begin{remark}\label{rem:subapartmentcondition}
We elaborate on the hypothesis of \cref{lem:insidesymplecticpartmentd0d1} by providing an example where it holds. Let $\rho\colon G\to \GL(V\phi)$ be a symplectic $p$-adic representation that stabilizes a $G$-stable primitive lattice, and with residual factors $\rho_1,\dots,\rho_d$ that are absolutely irreducible and pairwise nonisomorphic. Let $I=\{1,\dots,d\}$. We show later that there exists a symplectic basis $\mathrm{B}$ of $V$ with respect to which every $G$-stable lattice in $\Delta_n$ lies in the apartment $\Sigma(\mathrm{B})$ and is of the form $[a_1,\dots,a_{2n}]$ (see \cref{prop:Srhosymliesinanapartment}). Moreover, if $\Lambda_0=[0,\dots,0]$, then the residual representation $\rhobar_{\Lambda_0}$ is block-upper triangular with respect to a basis as in \cref{thm:uppertirangluarwithrespecttog0inGwithdistincteigenvaluesmodp}. Since $\rho_j$ is irreducible for all $j\in I$, there exists a partition $\mathcal{P}=\{B_j\}_{j\in I}$ of $\{1,\dots,2n\}$, satisfying the hypothesis of \cref{lem:insidesymplecticpartmentd0d1}, such that for all $j\in I$ and for all $i,k\in B_j$, we have $a_i=a_k$.
 \end{remark}

\subsection{Simplicial complex formed by stable vertices}


    Let $(V,\phi)$ be a symplectic $\Q_p$-vector space of dimension $2n$, 
     $\Delta_n$ denote the Bruhat--Tits building associated with $\GL(V,\phi)$, and 
    let $\rho\colon G\to \GL(V,\phi)$ be a symplectic representation that stabilizes a primitive lattice in $V$. Assume that $\rho$ is residually multiplicity free,  with residual factors  $\rho_1,\dots,\rho_d$ that are absolutely irreducible. Let $I=\{1,\dots,d\}$, and let  $n_i$ be the dimension of $\rho_i$ for all $i\in I$.

    \begin{definition}
        Let $S_\rho^\GSp$ be the set of $G$-stable vertices in $\Delta_n$. Let $\scrS_\rho^\GSp$ be the simplicial subcomplex formed by these vertices in $\Delta_n$.
        The \textit{dimension} of $\scrS_\rho^\GSp$ is defined to be the dimension of the maximal simplex contained in $\scrS_\rho^\GSp$. 
    \end{definition}
    Let $\sigma$ be the involution on $I$ such that $\rho_{\sigma(i)}\simeq \rho_i^\perp$, 
    $I_0=\{i\in I\ |\ \sigma(i)=i\}$, $I_1=\{i\in I\ |\ i<\sigma(i)\}$, and let $d_1=\# I_1$. 
    \begin{proposition}\label{prop:dimofSrhosym}
        The dimension of $\scrS_\rho^\GSp$ is $d_1$.
    \end{proposition}
    \begin{proof}
        Let $\Lambda_0$ be a $G$-stable primitive lattice. If $d_1=0$, then $\rhobar_{\Lambda_0}$ only has symplectic $G$-stable subspaces. This means that $[\Lambda_0]$ does not have any neighbours in $S_\rho^\GSp$. This implies that $\dim \scrS_\rho^\GSp=0=d_1$.

      Let $d_1\geq 1$. 
      There exists a maximal $G$-stable isotropic flag in $\Lambdabar_0$ by \cref{prop:maximalGstableisotropicflag}. That is, there exists lattices $\Lambda_1,\dots,\Lambda_{d_1}$ such that 
      \[
        p\Lambda_0\subset \Lambda_1\subset \dots \subset \Lambda_{d_1}\subset \Lambda_0,
      \]
      such that $\Lambda_{d_1}/p\Lambda_0$ is isotropic in $\Lambdabar_0$. This proves the result.
    \end{proof}
\subsubsection{Moving between stable symplectic lattices} 
Let $\rho\colon G\to \GL(V,\phi)$ be a symplectic representation such that $V$ has a $G$-stable primitive lattice. Suppose that $\rho$ is residually multiplicity free 
    and has residual factors  $\rho_1,\dots,\rho_d$ that are absolutely irreducible. Let $I=\{1,\dots,d\}$ and $I_0=\{i\in I\ |\ \sigma(i)=i\}$.

    In this section, we show the equivalence of the necessary assumption $I_0=\varnothing$ (for \cref{thm:ourmain}) and the existence of $G$-stable symplectic neighbours. The latter condition is crucial to our method of proof.

\begin{proposition}
    \label{thm:symplecticneighboursiffI0empty}
    There exist $G$-stable symplectic neighbours in $\Delta_n$ if and only if $I_0=\varnothing$. 
\end{proposition}
\begin{proof}
    Let $[\Lambda]$ and $[\Lambda_0]$ be $G$-stable symplectic neighbours in $\Delta_n$. Without loss of generality, we can assume that   
    $\Lambda_0$ is primitive and $p\Lambda_0\subset \Lambda\subset 
    \Lambda_0$ by \cref{cor:symneighoneofthemisprimitive}. By \cref{cor:symplecticneighboursofprimitive}, $\Lambda/p\Lambda_0$ is a Lagrangian subspace of $\Lambdabar_0$. By \cref{lem:extendingabasisofisotropictosymplectic}, there exists a basis of $\Lambdabar_0$ such that 
    \[
        \rhobar_{\Lambda_0}=\begin{pmatrix}
            \psi&*\\&\psi'
        \end{pmatrix}\quad\textnormal{and}\quad J\rhobar_{\Lambda_0} J^{-1}=\rhobar_\Lambda^\perp.
    \]
    This implies that $\psi'=\psi^\perp$, and that $I_0=\varnothing$, since $\rhobar_{\Lambda_0}$ has distinct residual factors.
    
    We prove the converse by using \cref{prop:dimofSrhosym}. If $I_0=\varnothing$, then $d=2d_1$ is even and the dimension of $\scrS_\rho^\GSp=d_1$. 
    For a $G$-stable primitive lattice $\Lambda_0$, there exists an isotropic $G$-stable flag 
    \[
        0=V_0\subset V_1\subset \dots \subset V_{d_1}\subset \Lambdabar_0.
    \]
    We see that $V_{d_1}$ is a Lagrangian subspace of $\Lambdabar_0$ since 
    \[
        2\dim V_{d_1}=2\sum_{i=1}^{d_1} \dim(V_{i}/V_{i-1})=2\sum_{i=1}^{d_1} n_i=\sum_{i=1}^{d_1} n_i +\sum_{i=1}^{d_1} n_{\sigma(i)}=2n.
    \]
    The $G$-stable Lagrangian subspace $V_{d_1}$ of $\Lambdabar_0$ corresponds to a $G$-stable symplectic neighbour of $\Lambda_0$ by \cref{cor:symplecticneighboursofprimitive}.
\end{proof}

\subsubsection{Contained in an apartment}
Let $\rho\colon G\to \GL(V,\phi)$ be a residually multiplicity-free symplectic representation stabilizing a primitive lattice in $V$. Let $\rho_1,\dots,\rho_d$ be the absolutely irreducible residual factors of $\rho$, $I=\{1,\dots,d\}$, and $\sigma\colon I\to I$ defined by $\rho_{\sigma(i)}\simeq \rho_i^\perp$. Let $I_0=\{i\in I\ |\ \sigma(i)=i\}$, $I_1=\{i\in I\ |\ i< \sigma(i)\}$, and let $d_j=\# I_j$ for $0\leq j\leq 1$.

 There exists a basis of $V$ with respect to which  $\scrS_\rho^\GL$ is contained in the associated $(d-1)$-dimensional apartment in $\mathcal{X}_d$ 
\cite[Proposition 3.1.1]{BellaicheGraftieaux}. In this section, using \cref{prop:contructingidempotents}, we show that $\scrS_\rho^\GSp$ is contained in an apartment of $\Delta_n$.   Moreover, using \cref{lem:insidesymplecticpartmentd0d1}, we show that $\scrS_\rho^\GSp$ is actually contained in a sub-apartment of dimension $d_0+d_1$.
\begin{proposition}\label{prop:Srhosymliesinanapartment}
    
    Let $\rho\colon G\to \GL(V,\phi)$ be a residually multiplicity-free symplectic $p$-adic representation that stabilizes a primitive lattice. 
    If there exists a basis of $V$ satisfying the conclusions of \cref{thm:uppertirangluarwithrespecttog0inGwithdistincteigenvaluesmodp}, then $\scrS_\rho^\GSp$ is contained in a sub-apartment of $\Delta_{n}$ of dimension $d_0+d_1$. 
\end{proposition}

\begin{proof}
    The strategy is the same as in \cite[Proposition 3.1.1]{BellaicheGraftieaux}. Let $\Lambda_0$ be a $G$-stable primitive lattice. Let $R=\Z_p[G]/\ker \rho$. By hypothesis, there exists a basis $\mathrm{B}$ of $\Lambda_0$ that satisfies the conclusions of \cref{thm:uppertirangluarwithrespecttog0inGwithdistincteigenvaluesmodp}. 
By \cref{prop:contructingidempotents}, there exist orthogonal idempotents $\{e_i\}_{i=1}^{2n}\subset R$, and a involution $\eta$ such that \eqref{item:idempotents1}, \eqref{item:idempotents2}, and \eqref{item:idempotents3} of \cref{prop:contructingidempotents} hold.
    We see that $\Lambda_0=\bigoplus_{i=1}^{2n} e_i\Lambda_0$ by \cref{prop:contructingidempotents} \eqref{item:idempotents2}. By \cite[Lemma 3.1.3]{BellaicheGraftieaux}, any $G$-stable lattice $\Lambda\subset \Lambda_0$ is of the form 
    $\Lambda=\bigoplus_{i=1}^{2n}p^{a_i}( e_i\Lambda_0)$ for some integers $a_1,\dots,a_{2n}$. 
    Let $\mathrm{B}_0$ be a symplectic basis of $\Lambda_0$ formed by permuting elements of $\mathrm{B}$.
    By \eqref{item:idempotents2} and \eqref{item:idempotents3} of \cref{prop:contructingidempotents}, there exists a symplectic basis of $\Lambda_0$ such that any $G$-stable lattice $\Lambda\subset \Lambda_0$ is given by $[a_1,\dots,a_{2n}]$, for some integers $a_1,\dots,a_{2n}$. This shows that $\scrS_\rho^\GSp$ is contained in the apartment of $\Delta_n$ corresponding to $\mathrm{B}_0$.

    To prove that $\scrS_\rho^\GSp$ is contained in a sub-apartment of dimension $d_0+d_1$, we use \cref{rem:subapartmentcondition}.
    Let $\mathcal{P}=\{B_j\}_{j=1}^d$ be the partition of $\{1,\dots,2n\}$ such that the idempotent $\mathbf{e}_j=\sum_{i\in B_j}e_i$ corresponds to the residual factor $\rho_j$ in the following sense: the map $R\to \mathbf{e}_jR\mathbf{e}_j\xrightarrow{\bar{\rho}_{\Lambda_0}} M_{2n}(\F_p)$ is given by 
    \[r\mapsto \begin{pmatrix}
            0&&&&\\ &\ddots&&&\\&&\rho_j(r)&&\\&&&\ddots&\\
            &&&&0
        \end{pmatrix}.\]
    With respect to the permutation $\sigma$ on $I$ such that $\rho_{\sigma(j)}\simeq \rho_j^\perp$, the hypothesis of \cref{lem:insidesymplecticpartmentd0d1} is satisfied, and the result follows.
\end{proof}

\subsubsection{Jordan--H\"older colouring} 
Let $\rho\colon G\to \GL(V)$ be a $p$-adic representation that is residually multiplicity-free with absolutely irreducible factors $\rho_1,\dots,\rho_d$. Let $\scrS_\rho^\GL$  be the subcomplex formed by $G$-stable vertices in the  building $\mathcal{X}_{2n}$ associated with $\GL(V)$.
We define the Jordan--H\"older colouring on the edges in $\scrS_\rho^\GL$ as in  \cite{BellaicheRibetslemmaPadova}. For a representation $\varphi\colon G\to \GL_n(\F_p)$, we let $\JH(\varphi)$ be the sum of the irreducible factors of $\varphi$ (with multiplicities) viewed as an element of the Grothendieck group of $\F_p[G]$. 

\begin{definition}\cite[\S 3.4.2]{BellaicheRibetslemmaPadova}
    Let $x=[\Lambda]\in S_\rho^\GL$, and let $y=[\Lambda']$ be a neighbour of $x$ in $S_\rho^\GL$ such that $p\Lambda\subset \Lambda'\subset \Lambda$. We let $\JH(x,y)\colonequals \JH(\Lambda'/p\Lambda)$. 
\end{definition}
    
    For $x_1,x_2\in S_\rho^\GL$ that are neighbours in $\mathcal{X}_{2n}$, we have $\JH(x_2,x_1)+\JH(x_1,x_2)=\sum_{i=1}^d \rho_i$ \cite[Proposition 3.4.3 (c)]{BellaicheRibetslemmaPadova}.
    Suppose now that $\rho\colon G\to \GL(V,\phi)$ is symplectic and residually multiplicity-free. The above colouring on the edges in $\scrS_\rho^\GL$ restricts to a colouring on the edges of $\scrS_\rho^\GSp$.
    For a symplectic vertex $x\in S_\rho^\GSp$, and $y\in S_\rho^\GSp$ neighbour of $x$ in $\Delta_n$, there exists $\{\varepsilon_i\}_{i=1}^{d}\subset \{0,1\}$ such that $\varepsilon_i+\varepsilon_{\sigma(i)}\leq 1$ for all $i\in I$, and $\JH(x,y)=\sum_{i=1}^{d}\varepsilon_i \rho_i$. The condition $\varepsilon_i+\varepsilon_{\sigma(i)}\leq 1$ is equivalent to the subspace corresponding to $y$ being isotropic. 

\begin{example}\label{example:twoandfourJordansymplecticneighbours} If $I_0=\varnothing$, then we have $2^{d_1}$-colourings of edges between $G$-stable symplectic vertices. 
When $d_1=1$, we have the following two colourings: 
    
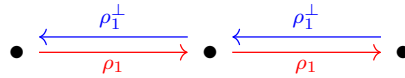
\[
        \begin{tikzcd}
           \dots \bullet\arrow[rr, red, "\rho_1", swap]&&\bullet\arrow[rr, red, "\rho_1", swap]\arrow[ll, blue, shift right=2, "\rho_1^\perp",swap]&&\bullet\arrow[ll, blue, shift right=2,"\rho_1^\perp", swap] \cdots 
        \end{tikzcd}
    \] 
In the above diagram,  $\bullet$ denotes a symplectic $G$-stable vertex
\end{example}
We end this section with a crucial result of \cite{BellaicheRibetslemmaPadova} that is used in the next section to prove the existence of symplectic vertices in $\scrS_\rho^\GSp$ with no edges of a specific colour.
Before we state the result, we define when two edges are \text{parallel} in an apartment. Let $\Sigma$ be an apartment with respect to a basis in $\mathcal{X}_{2n}$, and let $x_i=[a_1^{(i)},\dots,a_{2n}^{(i)}]$ for $1\leq i\leq 4$. The edge $(x_1,x_2)$ is \textit{parallel} to $(x_3,x_4)$ if $a_j^{(2)}-a_j^{(1)}=a_j^{(4)}-a_j^{(3)}$ for all $1\leq j\leq 2n$.
 
\begin{proposition} \cite[Propositions 3.4.3 (a) and 3.4.4]{BellaicheRibetslemmaPadova}.
    Let $\Sigma$ be an apartment in $\mathcal{X}_{2n}$ containing $\scrS_\rho^\GL$. Let $x,y,x',y'\in S_\rho^\GL$ such that $x,y$ and $x',y'$ are neighbours in $\mathcal{X}_{2n}$. Then $\JH(x,y)=\JH(x',y')$ if and only if the edges $(x,y)$ and $(x',y')$ are parallel in $\Sigma$. 
\end{proposition}

\subsubsection{Neighbours and extensions} In this section, we show the existence of certain edges in the symplectic extension graph associated with an irreducible symplectic $p$-adic representation. This is used crucially in our proof of \cref{thm:ourmain}. 

\begin{theorem}
    \label{prop:edgefromrhoitorhoji=1j=0}
    Let $\rho\colon G\to \GL(V,\phi)$ be an irreducible symplectic $p$-adic representation that admits a $G$-stable primitive lattice and with absolutely irreducible residual factors $\rho_1,\dots,\rho_d$. Let $I=\{1,\dots,d\}$ and let $\sigma$ be the involution on $I$ defined by $\rho_{\sigma(i)}\simeq \rho_i^\perp$ for all $i\in I$.
     Assume $\rho$ satisfies \eqref{assume:resmultfree} and \eqref{assum:distincteigenvaluesmodp}, and that $\sigma(i)\neq i$ for all $i\in I$.
    Given integers $\varepsilon_i\geq 0$ for all $i\in I$ such that $\varepsilon_i+\varepsilon_{\sigma(i)}\leq 1$, there exists a symplectic vertex $x=[\Lambda]\in S_\rho^\GSp$ such that $\Lambda/p\Lambda$ admits a nonsplit extension of $\rho_i$ by $\rho_j$, satisfying $\varepsilon_i=1$ and $\varepsilon_j=0$, as a subquotient.
\end{theorem}
We start with a result on the existence of triangles in $\scrS_\rho^\GSp$ with two symplectic vertices.
\begin{lemma}
    \label{prop:GstableinGstablesymplectic}
    Let $\rho\colon G\to \GL(V,\phi)$ be a $p$-adic symplectic representation that admits a $G$-stable primitive lattice, satisfies \eqref{assume:resmultfree}, and let $\rho_1,\dots,\rho_d$ be its absolutely irreducible residual factors. Let $I=\{1,\dots,d\}$ and assume that  $\sigma(i)\neq i$ for all $i\in I$.
    Let $x\in S_\rho^\GSp$ be a  symplectic vertex and let $y\in S_\rho^\GL$ be a neighbour of $x$ such that $\JH(x,y)=\sum_{i\in I}\varepsilon_i\rho_i$. Then $\varepsilon_i+\varepsilon_{\sigma(i)}\leq 1$ for all $i\in I$ if and only if $y$ is a neighbour of $x$ in $\Delta_n$. If $y$ is a neighbour of $x$ in $\Delta_n$, then there exists a symplectic vertex $z\in S_\rho^\GSp$ neighbouring $x$ and $y$ such that $\JH(x,z)=\JH(x,y)+\JH(y,z)$.
\end{lemma}
\begin{proof}
    Let $x=[\Lambda]$.
    We see that $y$ corresponds to a $G$-stable isotropic subspace $\Lambda'/p\Lambda$ of $\Lambda/p\Lambda$ if and only if $\varepsilon_i+\varepsilon_{\sigma(i)}\leq 1$ for all $i\in I$ by \cref{prop:psikispsimperp}. If $\Lambda'/p\Lambda$ is isotropic, then it is contained in a $G$-stable Lagrangian subspace $\Lambda''/p\Lambda$ of $\Lambda/p\Lambda$ by \cref{prop:dimofSrhosym}. This Lagrangian subspace of $\Lambda/p\Lambda$ corresponds to a symplectic neighbour $z$ of $x$ by \cref{cor:symplecticneighboursofsymplecticlattices}. Since $p\Lambda\subset \Lambda'\subset \Lambda''\subset \Lambda$, by definition it follows that $\JH(x,z)=\JH(x,y)+\JH(y,z)$.
\end{proof}
The nonexistence of $G$-stable neighbours in $\scrS_\rho^\GL$ with a certain $\JH$-colouring is equivalent to the existence of certain edges in $\Gamma_\rho^\GL$ by \cite[Proposition 3.4.5]{BellaicheRibetslemmaPadova}. 
 We have the following analogue in our setting:
 \begin{proposition}\label{prop:equivnonexistneighexitstext}
     Let $x\in S_\rho^\GSp$ and let $\Lambda$ be a symplectic lattice such that $x=[\Lambda]$. For all $i\in I$, let $\varepsilon_i\geq 0$ be an integer such that $\varepsilon_i+\varepsilon_{\sigma(i)}\leq 1$.  The following statements are equivalent:
     \begin{enumerate}[$(a)$]
         \item \label{item:nonexistneigh} There does not exist a neighbour $y\in S_\rho^\GSp$ of $x$ such that $\JH(x,y)=\sum_i \varepsilon_i \rho_i$.
         \item \label{item:existsexten} $\Lambda/p\Lambda$ admits a nonsplit extension of some $\rho_i$ by $\rho_j$ with $\varepsilon_i=1$ and $\varepsilon_j=0$.
     \end{enumerate}
 \end{proposition}
 \begin{proof}
     Suppose \eqref{item:nonexistneigh} is false. Then there exists neighbour $y$ of $x$ in $S_\rho^\GSp$ such that $\JH(x,y)=
    \sum_{i\in I}\varepsilon_i\rho_i$. Since $y$ is also a neighbour of $x$ in $S_\rho^\GL$, \eqref{item:existsexten} is false by \cite[Proposition 3.4.5]{BellaicheRibetslemmaPadova}. This shows that \eqref{item:existsexten} implies \eqref{item:nonexistneigh}. Suppose \eqref{item:existsexten} is false. Then every subquotient of $\Lambda/p\Lambda$ that is an extension of $\rho_i$ by $\rho_j$ with $\varepsilon_i=1$ and $\varepsilon_j=0$  is split. There exists a neighbour $y\in S_\rho^\GL$ of $x$ such that $\JH(x,y)=\sum_{i\in I}\varepsilon_i\rho_i$ by \cite[Proposition 3.4.5]{BellaicheRibetslemmaPadova}. 
    Since $\varepsilon_i+\varepsilon_{\sigma(i)}\leq 1$, the $G$-stable subspace of $\Lambdabar$ corresponding to $y$ is isotropic by \cref{prop:GstableinGstablesymplectic}. That is, $y\in S_\rho^\GSp$. This implies that \eqref{item:nonexistneigh} implies \eqref{item:existsexten}. 
 \end{proof}

The final ingredient in proving \cref{prop:edgefromrhoitorhoji=1j=0} is the result that guarantees the existence of symplectic vertices with no edges of a specific colour.

\begin{proposition}
\label{prop:existssymplecticwithnoneighbourswithspecificJH}
 Assume that $\rho$ is irreducible, satisfies \eqref{assume:resmultfree} and \eqref{assum:distincteigenvaluesmodp}, and that $I_0=\varnothing$.
    For all $i\in I$, let $\varepsilon_i\geq 0$ be an integer such that  
    $\varepsilon_i+\varepsilon_{\sigma(i)}\leq 1$.
    There exists a symplectic vertex $x\in S_\rho^\GSp$ with no neighbours $y\in S_\rho^\GSp$ with $\JH(x,y)=\sum_{i\in I}\varepsilon_i \rho_i$.
\end{proposition}

\begin{proof}
    Without loss of generality, we assume $\varepsilon_1=1$ and $\varepsilon_i=0$ for all $2\leq i\leq d$. By \cref{prop:Srhosymliesinanapartment},  $\scrS_\rho^\GSp$ lies inside an apartment $\Sigma$, corresponding to a symplectic basis $\mathrm{B}$, of $\Delta_{d_1}$. That is, every vertex in $\scrS_\rho^\GSp$ is of the form $[a_1,\dots,a_n,b_1,\dots,b_n]$.
    Let $x_0=[0,\dots,0]\in S_\rho^\GSp$ be the origin of $\Sigma$. 

    We prove by contradiction.
    Suppose that every symplectic vertex $x\in S_\rho^\GSp$ has a neighbour $y\in S_\rho^\GSp$ such that $\JH(x,y)=\rho_1$. Let $y_0$ be the neighbour of $x_0$ such that $\JH(x_0,y_0)=\rho_1$. Since $\rho_1$ is irreducible, we see that $y_0=[1,\dots,1,0,1\dots,1]$ where all but one entry is $1$. Without loss of generality, by permuting the elements of $\mathrm{B}$, we assume $y_0=[0,1,\dots, 1]$. Since $\rho$ is irreducible, the set $S_\rho^\GSp$ is finite by \cite[Proposition 3.2.1]{BellaicheRibetslemmaPadova}. Thus, there exists an integer $\mu$, maximal such that 
    \[
        x_\mu=[0,a_{2},\dots,a_{d_1},\mu,b_{2},\dots,b_{d_1}]\in S_\rho^\GSp
    \]
    is symplectic. By our assumption, there exists a neighbour $y_\mu\in S_\rho^\GSp$ of $x_\mu$ such that $\JH(x_\mu,y_\mu)=\rho_1$. By \cite[Proposition 3.4.4]{BellaicheRibetslemmaPadova}, since $y_0=[0,1\dots,1]$, we should have \[y_\mu=[0,a_2+1,\dots,a_{d_1}+1,\mu+1,b_2+1,\dots,b_{d_1}+1].\]
    By \cref{prop:GstableinGstablesymplectic}, there exists a symplectic vertex $z\in S_\rho^\GSp$, neighbouring $x_\mu$, such that 
    \[
        z=[0,\dots, \mu+1,\dots].
    \]
    This contradicts maximality of $\mu$.
\end{proof}

    \begin{proof}[Proof of \cref{prop:edgefromrhoitorhoji=1j=0}]
    Follows from \cref{prop:existssymplecticwithnoneighbourswithspecificJH} and \cref{prop:equivnonexistneighexitstext}.
\end{proof}

\section{Symplectic extension graph}
\label{sec:symextgraph}

In this section, we prove our main results: \cref{thm:ourmain} and \cref{thm:examples}.
We end by discussing an example where the symplectic extension graph is empty.
Throughout, our setup is as follows. 
Let $(V,\phi)$ be a symplectic $\Q_p$-vector space, and let $\rho\colon G\to \GL(V,\phi)$ be a symplectic representation that admits a $G$-stable primitive lattice.

\begin{remark}
     An important class of representations satisfying our hypothesis is continuous $p$-adic irreducible symplectic representations of a compact group $G$. Indeed, from the Bruhat--Tits fixed point theorem \cite[Theorem 11.23]{AbramenkoBrown} and \cite[Theorem 11.34]{AbramenkoBrown}, the maximal compact subgroups of $\GSp_{2n}(\Q_p)$ are stabilizers of vertices in the building and are conjugate over $\GL_{2n}(\Q_p)$ to $\GSp_{2n}(\Z_p)$. Therefore, by conjugating $\rho$ if necessary, we can assume that $\rho$ admits a $G$-stable primitive lattice.
\end{remark}

\subsection{Proof of \texorpdfstring{\cref{thm:ourmain}}{Theorem 1.6}}
We prove our main theorem in parts.  Let $\rho_1,\dots, \rho_d$ be the absolutely irreducible classes of residual factors of $\rho$. Let $I=\{1,\dots,d\}$, $\sigma:I\to I$ be defined such that $\rho_{\sigma(i)}\simeq \rho_i^\perp$, and let $I_0=\{i\in I\ |\ \sigma(i)=i\}$.
\begin{proposition}
    [\cref{thm:ourmain}\eqref{item:main2}]\label{lem:vertexiscontainedinaloop}
Assume $\rho$ is irreducible, satisfies \eqref{assume:resmultfree} and \eqref{assum:distincteigenvaluesmodp}, and $I_0=\varnothing$. For every $i\in I$, there is a path from $\rho_i$ to $\rho_{\sigma(i)}$ in $\Gamma_\rho^\GSp$.
\end{proposition}
\begin{proof}
    Let $i\in I$ and let $\Psi$ be the set of vertices in $\Gamma_\rho^\GSp$ to which there is a directed path from $\rho_i$. 
    For $ j\in I$, let 
    \[
        \varepsilon_j=\begin{cases}
            1&\textnormal{if }\rho_j\in \Psi\cup \{\rho_i\},\\
            0&\textnormal{else.}
        \end{cases}
    \] 
    If $\varepsilon_j=\varepsilon_{\sigma(j)}=1$ for some $j\in I$, then we are done by \cref{prop:skeqsymgraph}. Indeed,
     if there is a path from $\rho_i$ to $\rho_j$ and $\rho_{\sigma(j)}$, then there is a path from $\rho_i$ to $\rho_{\sigma(i)}$.
     Suppose for all $j\in I$, we have $\varepsilon_j+\varepsilon_{\sigma(j)}\leq 1$, then there exists an edge from $\rho_j$ with $\varepsilon_j=1$ to $\rho_k$ with $\varepsilon_k=0$ in $\Gamma_\rho^\GSp$, by \cref{prop:edgefromrhoitorhoji=1j=0}. A path from $\rho_i$ to $\rho_j$ exists in $\Gamma_\rho^\GSp$ and gives a path from $\rho_i$ to $\rho_k$. This is a contradiction.
\end{proof}

\begin{theorem}[\cref{thm:ourmain}\eqref{item:main1}]\label{thm:mainI_0empty}
    Assume $\rho$ is irreducible, satisfies \eqref{assume:resmultfree} and \eqref{assum:distincteigenvaluesmodp}, and $I_0=\varnothing$.
    The directed graph $\Gamma_\rho^\GSp$ has at most $d_1$ strongly connected components. Moreover, when $\Gamma_\rho^\GSp$ has $d_1$ strongly connected components, then $\Gamma_\rho^\GSp=\bigsqcup_{i<\sigma(i)}\Gamma_i$, where $\Gamma_i$ is the strongly connected directed graph on $\rho_i$ and $\rho_{\sigma(i)}$.
\end{theorem}

\begin{proof}
Let $i\in I$.
For a vertex $\rho_i$ of $\Gamma_\rho^\GSp$, let $\Gamma_i$  be the subgraph of $\Gamma_\rho^\GSp$ whose vertices are either  $\rho_i$, $\rho_{\sigma(i)}$, or has a path to and from $\rho_i$ and $\rho_{\sigma(i)}$. We claim that 
  \begin{enumerate}[(i)]
        \item \label{item:gammapsistronglyconn} $\Gamma_{i}=\Gamma_{\sigma(i)}$ and it is a strongly connected subgraph of $\Gamma_\rho^\GSp$;
        \item \label{item:shapeofgammapsi0} $\rho_j\in \Gamma_{i}$ if and only if $\rho_{\sigma(j)}\in \Gamma_{i}$ for any $j\in I$. 
    \end{enumerate}
    We note that \eqref{item:shapeofgammapsi0} follows from \cref{prop:skeqsymgraph}.
    Let us prove \eqref{item:gammapsistronglyconn}. By definition, it is immediate that $\Gamma_{i}=\Gamma_{\sigma(i)}$. If $\rho_j$ and $\rho_k$ belong to $\Gamma_{i}$, then we see that there is a path from $\rho_j$ to $\rho_k$:

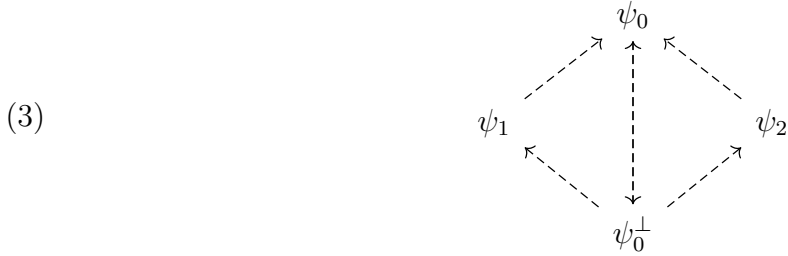
\begin{equation}\label{eq:pathpsi1psi2}
    \begin{tikzcd}
                          & \rho_i \arrow[dd, dashed]                                             &                           \\
\rho_j \arrow[ru, dashed] &                                                                       & \rho_k \arrow[lu, dashed] \\
                          & \rho_{\sigma(i)} \arrow[uu, dashed] \arrow[lu, dashed] \arrow[ru, dashed] &                          
\end{tikzcd}
\end{equation}
In \eqref{eq:pathpsi1psi2}, the dashed line represents a path. This proves \eqref{item:gammapsistronglyconn}.   
For all $i,j\in I$ and $i\neq j$, if $\Gamma_i\neq \Gamma_j$, then we claim that there is no edge from $\Gamma_i$ to $\Gamma_{j}$. If there is an edge from $\Gamma_i$ to $\Gamma_{j}$, then there is a path from $\rho_i$ to $\rho_j$. There is a path from $\rho_{\sigma(j)}$ to $\rho_{\sigma(i)}$ by \cref{prop:skeqsymgraph}. Thus, we see that $\rho_i$ has a path to and from $\rho_j$ and $\rho_{\sigma(j)}$ by \cref{lem:vertexiscontainedinaloop}. That is, $\rho_i\in \Gamma_{j}$ and $\Gamma_i=\Gamma_j$. This shows that $\Gamma_\rho^\GSp$ is the disjoint union of strongly connected subgraphs of the form $\Gamma_i$. 

Since $\Gamma_i$ contains the two vertices $\{\rho_i,\rho_{\sigma(i)}\}$, the maximum number of strongly connected components is $d_1=d/2$. Moreover, $\Gamma_\rho^\GSp$ has $d_1$ connected components  if and only if $\Gamma_i$ is the strongly connected graph on $\{\rho_i,\rho_{\sigma(i)}\}$.
\end{proof}

\subsection{Two Jordan--H\"older factors}

Let $(V,\phi)$ be a symplectic $\Q_p$-vector space, and let $\rho:G\to \GL(V,\phi)$ be a symplectic representation admitting a $G$-stable primitive lattice.
When $\rho$ has two residual factors, then we have a full characterization of $\Gamma_\rho^\GSp$.
\begin{theorem}\label{thm:twoJordanHolderfactorsclassification}
    Assume $\rho$ is irreducible and satisfies \eqref{assum:distincteigenvaluesmodp} and \eqref{assume:resmultfree}.
    The symplectic extension graph $\Gamma_\rho^\GSp$ associated with $\rho$ is strongly connected if and only if $I_0=\varnothing$. 
\end{theorem}

The \textit{if} part follows from \cref{thm:ourmain} when $d=2$. The \textit{only if} part follows from the result below.
\begin{proposition}\label{prop:neighboursofsymplecticlattice2JHfactors}
    Let $I_0=\{1,2\}$, and let $\JH(\rhobar^{\mathrm{ss}})=\rho_1+\rho_2$. For any symplectic $G$-stable lattice $\Lambda$, $\rhobar_\Lambda$ is semisimple.  
\end{proposition}
\begin{proof}
    By \cref{prop:skeqsymgraph}, if $[\Lambda]\in S_\rho^\GSp$ is symplectic such that $\rhobar_\Lambda$ is a nonsplit extension of $\rho_2$ by $\rho_1$, then it is a nonsplit extension of $\rho_1$ by $\rho_2$. 
\end{proof}

\begin{example} \label{example:2JHdisconnected} 
We provide an example where the symplectic extension graph is disconnected and a proper subgraph of Bella\"iche's connected extension graph.

Let $G$ be the elements in $\GSp_4(\Z_p)$ whose reduction modulo $p$ is the set of matrices $\begin{pmatrix}
A&B\\C&D
\end{pmatrix}$ with $A,B,C,D$ are diagonal modulo $p$; that is, \[
G\bmod{p}=\GL_2(\F_p)\times_{\GL_1(\F_p)}\GL_2(\F_p)\subset \GSp_4(\F_p).\] More precisely, 
\[
	G=\left\{\begin{pmatrix}
	a_1&pa_2&b_1&pb_2\\
	pa_3&a_4&pb_3&b_4\\
		c_1&pc_2&d_1&pd_2\\
	pc_3&c_4&pd_3&d_4\\
	\end{pmatrix}\ \big{|}\ a_1d_1-b_1c_1\equiv a_4d_4-b_4c_4\bmod{p}\right\}.
\] 
Let the representation $\rho:G\hookrightarrow \GSp_4(\Z_p)$ be the inclusion map. Then $\rho$ is irreducible, and $\rhobar=\rho_1\oplus \rho_2$, where \(
	\rho_i=\begin{pmatrix}
	a_i&b_i\\c_i&d_i
	\end{pmatrix}\bmod{p}
\) is absolutely irreducible, for $1\leq i\leq 2$, and $\det \rho_1=\det \rho_2$.
To show $\rho_1\not\simeq \rho_2$, we show that $\Tr\rho_1\neq \Tr\rho_2$. Indeed, 
for \[g=\begin{pmatrix}
1&&&\\
&&&-1\\
&&1&\\
&1&&
\end{pmatrix},\]
we see that $\Tr\rho_1(g)=2$ and $\Tr\rho_2(g)=0$.
Hence, the irreducible representation $\rho$ is residually reducible and multiplicity-free. 
Moreover, $\rho_1\rho_2^t$ is not a character as $(\rho_1\rho_2^t)(g)$ is not a scalar. By \cref{prop:neighboursofsymplecticlattice2JHfactors}, $\Gamma_\rho^\GSp$ is disconnected.
\end{example}

\subsection{Four Jordan--H\"older factors}

Let $(V,\phi)$ be a symplectic $\Q_p$-vector space, and let $\rho \colon G\to \GL(V,\phi)$ be a symplectic representation admitting a $G$-stable symplectic lattice and has residual factors $\rho_1, \rho_2,\rho_3,$ and $\rho_4$ that are absolutely irreducible and pairwise nonisomorphic. 
Let $I=\{1,2,3,4\}$ and let $I_0=\{i\in I\ |\ \sigma(i)=i\}$.

\begin{theorem}[\cref{thm:mainI_0empty}: $d_1=2$]\label{thm:Bellaichegeneralization4JH} Assume $\rho$ is irreducible, satisfies \eqref{assum:distincteigenvaluesmodp} and \eqref{assume:resmultfree}, and assume that $I_0=\varnothing$.
    The directed graph $\Gamma_\rho^\GSp$ is either strongly connected, or $\Gamma_\rho^\GSp=\Gamma_1\sqcup \Gamma_2$, where $\Gamma_i$ is the strongly connected graph on $\{\rho_i,\rho_i^\perp\}$ for $1\leq i\leq 2$.
\end{theorem}
\begin{example}\label{exm:JH4graph}
    We illustrate examples for the two cases of \cref{thm:Bellaichegeneralization4JH}: one where the symplectic extension graph is strongly connected and one where it is a union of two strongly connected subgraphs. We generalize this example in the next section to show that \cref{thm:ourmain} is optimal (see \cref{thm:examplesformaintheorem}).
    
    We first provide an example where the symplectic extension graph is strongly connected.  
    Let 
    \[
        G_1=\left\{\begin{pmatrix}
            a_{11}&pa_{12}&a_{13}&pa_{14}\\pa_{21}&a_{22}&pa_{23}&a_{24}\\
            pa_{31}&pa_{32}&a_{33}&pa_{34}\\
            pa_{41}&pa_{42}&pa_{43}&a_{44}
        \end{pmatrix}\in \GSp_4(\Z_p)\ \Big{|}\ a_{ij}\in \Z_p\textnormal{ for all }1\leq i,j\leq 4 \right\}.
    \]
    Let $\rho\colon G_1\hookrightarrow \GSp_{4}(\Z_p)$ be the inclusion. It is not hard to see that $\rho$ is absolutely irreducible (or equivalently, $\rho(\Q_p[G_1])=M_4(\Q_p)$; see \cref{lem:inclusionisirred} for a proof) and residually reducible with four Jordan--H\"older factors: 
    \[
        \rho_1\colonequals a_{11}\bmod p, \rho_2\colonequals a_{22}\bmod p\quad \textnormal{and}\quad \rho_1^\perp\colonequals a_{33}\bmod p, \rho_2^\perp\colonequals a_{44}\bmod p.
    \] 
    The diagonal matrices in $\GSp_{4}(\Z_p)$ lie in $G_1$. Let $\mathrm{B}$ be the symplectic basis fixed. Since $\mathrm{B}$ satisfies the conclusions of \cref{thm:uppertirangluarwithrespecttog0inGwithdistincteigenvaluesmodp}, $\scrS_\rho^\GSp$ lies in the apartment $\Sigma$ associated with $\mathrm{B}$ of $\Delta_2$ by \cref{prop:Srhosymliesinanapartment}.
 By construction, the origin $x_0=[0,\dots,0]$ of $\Sigma$ is $G_1$-stable and primitive.
    
    Any symplectic vertex $x$ in $\Sigma$ is of the form $x=[\alpha_1,\alpha_2,\beta_1,\beta_2]$ with $\alpha_1+\beta_1=\alpha_2+\beta_2$ by \cref{prop:symlatticescharacterizationviasymbasis}. If $x$ is a symplectic $G_1$-stable vertex, then  
    for every $g\in G_1$, we have
    \[
        \mathrm{diag}(p^{\alpha_1},p^{\alpha_2},p^{\beta_1},p^{\beta_2})\rho(g)
        \mathrm{diag}(p^{-\alpha_1},p^{-\alpha_2},p^{-\beta_1},p^{-\beta_2})\in \GSp_{4}(\Z_p).
    \]
    Since $g\in G_1$ can be chosen such that $a_{31},a_{42}$ do not lie in $p\Z_p$, we get
    \[
        0\leq \alpha_1-\beta_1\leq 1\quad \textnormal{and}\quad 0\leq \alpha_2-\beta_2  \leq  1.
    \]
    We claim either $x=x_0$ or $x=x_1\colonequals [1,1,0,0]$.
    We see that $\alpha_1=\beta_1$ if and only if $\alpha_2=\beta_2$. If $\alpha_i=\beta_i$ and $\alpha_j=\beta_j+1$ for $j\neq i$, then 
    \[
        2\alpha_i=\alpha_i+\beta_i=\alpha_j+\beta_j=2\alpha_j+1,
    \]
    a contradiction. If $\alpha_i=\beta_i$ for some $i\in \{1,2\}$, then $\alpha_1=\beta_1=\alpha_2=\beta_2$ as $x$ is symplectic. Therefore, $x=x_0$.
    If $\alpha_1=1+\beta_1$, then $\alpha_2=1+\beta_2$.
    Using $\alpha_1+\beta_1=\alpha_2+\beta_2$, we see that  
    $\alpha_1=\alpha_2=1$ and $\beta_1=\beta_2=0$. This shows that $x=x_1$. 
    
    Therefore, $x_0$ has a unique $G_1$-stable neighbour $x_1=[1,1,0,0]$.
    The vertex $x_0$ realizes a nonsplit extension of $\rho_1^\perp$ by $\rho_1$, and a nonsplit extension of $\rho_2^\perp$ by $\rho_2$.
    Moreover, $x_1$ realizes a nonsplit extension of $\rho_1$ by $\rho_1^\perp$ and $\rho_2^\perp$, and $\rho_2$ by $\rho_2^\perp$ and $\rho_1^{\perp}$.    
    The following diagram depicts $\scrS_\rho^\GSp$ in $\Sigma$:
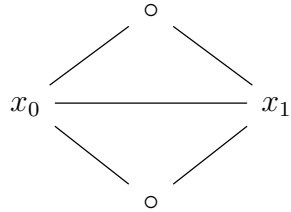
    \[
       \begin{tikzcd}
                                                                & \circ  \arrow[rd, no head] &     \\
x_0 \arrow[ru, no head] \arrow[rd, no head] \arrow[rr, no head] &                           & x_1 \\
                                                                & \circ \arrow[ru, no head] &    
\end{tikzcd}
    \]
    In the above diagram, $\circ$ denotes a neighbouring vertex in $\scrS_\rho^\GSp$ of the symplectic vertices $x_0$ and $x_1$.
In this case, the symplectic extension graph $\Gamma_\rho^\GSp$ is 
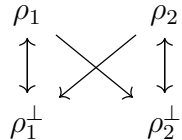
\[
    \begin{tikzcd}
\rho_1 \arrow[d]\arrow[dr]       & \rho_2 \arrow[d]\arrow[dl]      \\
\rho_1^\perp \arrow[u] & \rho_2^\perp \arrow[u]
\end{tikzcd}
\]

We provide an example where $\Gamma_\rho^\GSp$ has two strongly connected components.
Let \[G_2=\left\{\begin{pmatrix}
    a_{11}&pa_{12}&a_{13}&pa_{14}\\pa_{21}&a_{22}&pa_{23}&a_{24}\\
            pa_{31}&p^2a_{32}&a_{33}&pa_{34}\\
            p^2a_{41}&pa_{42}&pa_{43}&a_{44}
        \end{pmatrix}\in G\right\}.\]
        We see that $G_2$ is a subgroup of $G$. Let $\rho'\colonequals \rho|_{G_2}$, $\rho_i'\colonequals \rho_i|_{G_2}$ and let $\rho_i'^\perp\colonequals \rho_i^\perp|_{G_2}$ for all $1\leq i\leq 2$. Then $\rho'$ is an irreducible representation with residual factors $\rho_1',\rho_2',\rho_1'^{\perp},$ and $\rho_2'^\perp$. By the same calculations as above, there exist two $G_2$-stable symplectic vertices $x_0$ and $x_1$, and by the shape of $G_2$, the edges from $\rho_2$ to $\rho_1^\perp$, and $\rho_1$ to $\rho_2^\perp$ are not realized by $\rhobar'_{x_1}$. In this case, the symplectic extension graph $\Gamma_{\rho'}^\GSp$ is 
         \[
    \begin{tikzcd}
\rho_1' \arrow[d]       & \rho_2' \arrow[d]      \\
\rho_1'^\perp \arrow[u] & \rho_2'^\perp \arrow[u]
\end{tikzcd}
\]
We note that although $\rho$ and $\rho'$ have the same image modulo $p$, their symplectic extension graphs are different.
\end{example}

\begin{remark}
    We mention certain consequences of \cref{thm:Bellaichegeneralization4JH}, under some assumptions, in the context of genus $2$ Klingen Eisenstein series congruences with a genus $2$ cuspidal Siegel eigenform. Such congruences have been studied in \cite{berger2025klingeneisensteinseriescongruences} with applications to certain modularity results. In \cite{berger2025klingeneisensteinseriescongruences},  they work with irreducible symplectic Galois representations that have two residual Jordan--H\"older factors (as in our hypothesis of \cref{thm:ourmain} when $d=2$) and whose corresponding symplectic extension graph is strongly connected. 
    When the Klingen Eisenstein congruence corresponds to an irreducible symplectic Galois representation that is residually given by four characters, then using \cref{thm:Bellaichegeneralization4JH}, we produce certain nontrivial cohomology classes.
    
    For an integer $r$, let $\F_p(r)$ be the $G_\Q$-module $\F_p$ with the $G_\Q$-action given by multiplication by the $r$-th power of the mod-$p$ cyclotomic character $\omega$. 
        Let $E_k$ denote the normalized Eisenstein series of weight $k\geq 4$ level $\SL_2(\Z)$, and let $f$ be a  cuspidal eigenform of weight $k$ and level $\SL_2(\Z)$ such that $f\equiv E_k\bmod p$.
        Let $F$ be a genus $2$ Siegel cuspidal eigenform that is eigenvalue-congruent modulo $p$ to the genus $2$ Klingen Eisenstein series associated with $f$. For simplicity, we assume that $F$ has Hecke eigenvalues in $\Z_p$. 
        Let $\rho\colon G_\Q\to \GSp_4(\Z_p)$ be the $p$-adic Galois representation associated with $F$. Then $\rhobar^{\mathrm{ss}}=1\oplus \omega^{k-1}\oplus \omega^{k-2}\oplus \omega^{2k-3}$ and $(\omega^{k-1})^\perp=\omega^{k-2}$.
        If $\rho$ is irreducible, then at least one of the following situations must occur by \cref{thm:Bellaichegeneralization4JH}.
    \begin{enumerate}[(a)]
        \item \label{item:Galoiscohomononconnectedcase} 
        When $\Gamma_\rho^\GSp$ is not strongly connected, then $H^1(G_\Q,\F_p(\pm 3\mp 2k))\neq 0\neq H^1(G_\Q,\F_p(\pm 1))$. 
        \item Suppose $\Gamma_\rho^\GSp$ is strongly connected. Then the following must hold
        \begin{enumerate}[(i)]
            \item $H^1(G_\Q,\F_p(3-2k))\neq 0$ or $ H^1(G_\Q,\F_p(2-k))\neq 0$ or $H^1(G_\Q,\F_p(1-k))\neq 0$. This follows from the existence of a path from $\omega^{2k-3}$ to $1$ in $\Gamma_\rho^\GSp$.
            \item $H^1(G_\Q,\F_p(2k-3))\neq 0$ or $ H^1(G_\Q,\F_p(k-2))\neq 0$ or $H^1(G_\Q,\F_p(k-1))\neq 0$. This follows from the existence of a path from  $1$ to $\omega^{2k-3}$ in $\Gamma_\rho^\GSp$.
            \item $H^1(G_\Q,\F_p(-1))\neq 0$ or $ H^1(G_\Q,\F_p(2-k))\neq 0$ or $H^1(G_\Q,\F_p(k-1))\neq 0$. This follows from the existence of path from $\omega^{k-2}$ to $\omega^{k-1}$ in $\Gamma_\rho^\GSp$.
            \item $H^1(G_\Q,\F_p(1))\neq 0$ or $ H^1(G_\Q,\F_p(1-k))\neq 0$ or $H^1(G_\Q,\F_p(k-2))\neq 0$. This follows from the existence of path from $\omega^{k-1}$ to $\omega^{k-2}$ in $\Gamma_\rho^\GSp$.
        \end{enumerate}
         
    \end{enumerate}
    Further study of the local shape of $\rho$ at $p$ should give us more information on the local conditions satisfied by the cohomology classes we produce.
\end{remark}

\subsection{Proof of \texorpdfstring{\cref{thm:examples}}{Theorem 1.7}}
In this section, we prove \cref{thm:examples} by generalizing \cref{exm:JH4graph} (see \cref{thm:examplesformaintheorem} below).  


Let $n\geq 1$ be an integer, and let $\sigma$ be the involution on $I_{2n}=\{1,\dots,2n\}$ defined by $\sigma(i)=i+n$ for $1\leq i\leq n$. Let $r$ be an integer fixed throughout this section such that $1\leq r\leq n$. Let $\{\eta_{i,j}\}_{1\leq i,j\leq n}$ satisfy the following conditions. 
    \begin{enumerate}
        \item For all $1\leq i\leq n$, let $\eta_{i,i}=\eta_{i,\sigma(i)}=\eta_{\sigma(i),\sigma(i)}=\eta_{\sigma(i),i}-1=0$.
        \item For all $1\leq i\neq j\leq n$, let
            $\eta_{i,j}=\eta_{i,\sigma(j)}=\eta_{\sigma(i),\sigma(j)}=1$.
        \item \label{item:etarcausingnonexistenceofedges} For all $1\leq i\leq r-1$, let  $\eta_{\sigma(i),j}=2=\eta_{\sigma(j),i}$ for every $1\leq  j\leq n$ with $j\neq i$; else $\eta_{\sigma(i),j}=\eta_{\sigma(j),i}=1$.
    \end{enumerate}
    Let us denote the assignment $(i,j)\mapsto \eta_{i,j}$ defined above by $\eta^{(r)}$.

    \begin{definitions}\cite[2.1.1 and 2.1.2]{BellaicheGraftieaux}
        A \textit{directed numbered graph} is a triple $(V,E,\eta)$, where $(V,E)$ is a directed graph with vertex set $V$ and edge set $E\subseteq V^2$, and $\eta:E\to \Z\sqcup\{\infty\}$ is a map. We say $(V,E,\eta)$ is a \textit{metric graph} if $E=V^2$, $\eta(v,v)=0$ for all $v\in V$, and $\eta(u,v)\leq \eta(u,w)+\eta(w,v)$ for all $u,v,w\in V$. 
    \end{definitions}
    
    \begin{lemma}\label{lem:metricgraphexample}
        The directed numbered graph $(I_{2n},I_{2n}^2,\eta^{(r)})$ is a metric graph.
    \end{lemma}
    \begin{proof}
        We need to show for all $i,j,k\in I_{2n}$, 
        $
            \eta_{i,j}\leq \eta_{i,k}+\eta_{k,j}.
        $
        The inequality is true for $i=j$ by definition. Suppose that $i\neq j$. We see that
        \[
            \eta_{i,k}+\eta_{k,j}=\begin{cases}
                \eta_{i,j}&\textnormal{if }k\in\{i,j\},\\
                \eta_{\sigma(i),j}&\textnormal{if }k\in\{\sigma(i),\sigma(j)\},\\
                \geq 2&\textnormal{otherwise.}
            \end{cases}
        \]
        Since $\eta_{i,j}=\eta_{\sigma(i),j}\leq 2$, the result follows.
    \end{proof}
    Let
        $G_r$ be the collection of elements $(a_{ij})\in \GSp_{2n}(\Z_p)$ with $a_{ij}\in p^{\eta_{i,j}}\Z_p$.
     We have the inclusions
    \[
        G_{n}\subset \dots \subset G_2\subset G_1.
    \]
    \begin{lemma}\label{lem:inclusionisirred}
        For any $1\leq r\leq n$, we have $\Q_p[G_r]=M_{2n}(\Q_p)$. 
    \end{lemma}
    \begin{proof}
     Let $E_{ij}$ be the matrix in $M_{2n}(\Q_p)$ with $1$ at the $(i,j)$-th entry and $0$ everywhere else. We show that $E_{ij}\in \Q_p[G_r]$ for all $1\leq i,j\leq 2n$. Since there exists a diagonal element $g_0\in G_r$ with residually distinct eigenvalues, $E_{i,i}\in \Q_p[G_r]$ for all $1\leq i\leq 2n$ by \cref{lem:Mgeneratesdiagonalmatrices}. Let $\sigma$ be the involution on $\{1,\dots,2n\}$ such that $\sigma(k)=k+n$ for all $1\leq k\leq 2n$.  Let $\mathrm{Id}\in M_{2n}(\Z_p)$ be the identity matrix. By a simple calculation, for all $1\leq i\leq 2n$, we see that $p^{\eta_{i,\sigma(i)}}E_{i,\sigma(i)}+\mathrm{Id}\in G_r$. This implies that $E_{i,\sigma(i)}\in \Q_p[G_r]$ for all $1\leq i\leq 2n$. It remains to show that $E_{i,j}\in \Q_p[G_r]$ for $1\leq i,j\leq 2n$ such that $j\not\in \{i,\sigma(i)\}$. Since $\eta_{i,j}=\eta_{\sigma(j),\sigma(i)}$, we see that $\mathrm{Id}+p^{\eta_{i,j}}(E_{i,j}-E_{\sigma(j),\sigma(i)})\in G_r$. This implies that  $E_{i,j}-E_{\sigma(j),\sigma(i)}\in \Q_p[G_r]$. Since $E_{i,i}\in \Q_p[G_r]$ and $E_{i,j}=E_{ii}(E_{i,j}-E_{\sigma(j),\sigma(i)})$, we see that $E_{i,j}\in \Q_p[G_r]$.
    \end{proof}
    \begin{proposition}
        \label{thm:examplesformaintheorem}
        The inclusion $\rho\colon G_r\hookrightarrow
        \GSp_{2n}(\Z_p)$ is an irreducible representation whose residual factors are nonisomorphic characters. The symplectic extension graph $\Gamma_\rho^\GSp$ has $r$ connected components.
    \end{proposition}
    \begin{proof}
     By \cref{lem:metricgraphexample} and \cite[Lemme 3.2.1]{BellaicheGraftieaux}, $G_r$ is a subgroup of $\GSp_{2n}(\Z_p)$. Moreover, the map  $\rho\colon G_r\hookrightarrow \GSp_{2n}(\Z_p)$ defined by inclusion is irreducible by \cref{lem:inclusionisirred}.
     
     For all $1\leq i\leq 2n$, the residual factors are 
    \[
        \chi_i((a_{jk}))=a_{ii}\bmod{p}.
    \]
    By definition, the residual characters $\chi_1,\dots,\chi_{2n}$ are pairwise distinct. 
    Moreover, if $\sigma(i)=i+1$ for $1\leq i\leq n$, then we see that $\chi_{\sigma(i)}=\chi_i^\perp$ for all $i\in I_{2n}$. That is, 
    if $\lambda:G_r\to \Z_p^\times$ is the similitude character of $\rho$, then $\chi_{\sigma(i)}\chi_i=\bar\lambda$.
    
    Let $\mathrm{B}_0$ be the symplectic basis fixed of $\Z_p^{2n}$.
    Since the diagonal matrices of $\GSp_{2n}(\Z_p)$ are in $G_r$, $\scrS_\rho^\GSp$ lies in the apartment of $\Delta_n$ associated with $\mathrm{B}_0$ by \cref{prop:Srhosymliesinanapartment}. 
    If $x=[\alpha_1,\dots,\alpha_n,\beta_1,\dots,\beta_n]$ is a $G$-stable symplectic vertex, then we see that $0\leq \alpha_i-\beta_i\leq 1$ for all $i\in \{1,\dots,2n\}$. This follows by the condition $\eta_{\sigma(i),i}=1$ for all $1\leq i\leq n$. 
    By a computation similar to the one done in \cref{exm:JH4graph}, we see the only $G$-stable symplectic vertices are $x_0=[\Lambda_0]=[0,\dots,0]$, and $x_1=[1,\dots, 1,0,\dots,0]$.  
    By our definition of $\eta^{(r)}$, in particular \eqref{item:etarcausingnonexistenceofedges}, the graph $\Gamma_\rho^\GSp$ is a disjoint union of $r$ strongly connected subgraphs: $r-1$ strongly connected subgraphs with vertices $\{\chi_i,\chi_{\sigma(i)}\}_{i=1}^{r-1}$, and a strongly connected subgraph on the remaining vertices.
    \end{proof}
\subsection{Three Jordan--H\"older: an example}
\label{subsec:3JH}
In this section, we provide an example of an irreducible symplectic $p$-adic representation $\rho$ for which $\Gamma_\rho^\GSp$ is an empty graph.
This contrasts with \cref{thm:ourmain}. An important distinction of this example from the one provided in \cref{example:2JHdisconnected} is that the latter is residually reducible over $\GL_4(\F_p)$ and not over $\GSp_4(\F_p)$. That is, the residual representation in \cref{example:2JHdisconnected} does not stabilize an isotropic subspace.

\begin{proposition}\label{prop:example3JH}
    There exists an irreducible symplectic $p$-adic representation $\rho\colon G\to \GSp_{4}(\Z_p)$, satisfying \eqref{assume:dist} and \eqref{assume:resmultfree},  with three residual factors such that 
   $\Gamma_\rho^\GSp$ is the empty graph and $\Gamma_\rho^\GL$ is the complete graph.
\end{proposition}
\begin{proof}
    Let $I_4=\{1,2,3,4\}$, and let $\mathcal{G}=(I_4,I_4^2,\eta)$ be the numbered graph defined by 
    \begin{enumerate}[(i)]
        \item $\eta_{i,i}=0$ for all $i\in I_4$,
        \item $\eta_{i,j}=1$ for all $1\leq i\neq j\leq 3$,
        \item $\eta_{4,1}=\eta_{4,3}=\eta_{1,4}=\eta_{,4}=1$ and $\eta_{4,2}=\eta_{2,4}=0$.
    \end{enumerate} 
    By definition, for all $i,j,k\in I_4$, $\eta_{ij}+\eta_{jk}\in \{0,1,2\}$. To show that $\eta$ is a metric, the only case to check is when $\eta_{ij}+\eta_{jk}=0$. In this case, it follows that $\eta_{ik}=0$ by definition.
    
    Let $G$ be the group of elements $(a_{ij})\in \GSp_{4}(\Z_p)$ with $a_{ij}\in p^{\eta_{i,j}}\Z_p$.
    It immediate that $G$ is in fact the preimage in $\GSp_{4}(\Z_p)$ under the reduction mod $p$ map of the Levi subgroup $\GL_1(\F_p)\times \GL_1(\F_p)\times \GL_2(\F_p)$ of the Klingen parabolic in $\GSp_{4}(\F_p)$.
    Let $\rho \colon G\hookrightarrow \GSp_4(\Z_p)$ be the inclusion. Then $\rho$ is irreducible, residually multiplicity free with three residual factors, and satisfies \eqref{assume:dist} by the definition of $G$.
    Let $\Sigma$ be the apartment in $\mathcal{X}_4$ with respect to the chosen symplectic basis.  By \cite[Proposition 3.1.1]{BellaicheGraftieaux}, 
    $\scrS_\rho^\GL$ lies in the sub-apartment of $\Sigma$ of dimension $2$ in $\mathcal{X}_{3}$ and,
    by \cite[Lemme 3.2.1]{BellaicheGraftieaux}, we have \[\scrS_\rho^\GL=\{[a_1,a_2,a_3,a_4]\in \Sigma\ |\ a_i-a_j\leq \eta_{i,j}\}.\]
    By our choice of $\eta$, the simplicial complex $\scrS_\rho^\GL$ is a triangulated hexagon in the sub-apartment of $\Sigma$. By \cite[Th\'eor\`eme 3.3.3]{BellaicheGraftieaux}, $\Gamma_\rho^\GL$ is the complete graph on the three residual factors. On the other hand,  we see that $S_\rho^\GSp=\{[0,0,0,0]\}$ by \cref{prop:symlatticescharacterizationviasymbasis}. If $\Lambda_0$ is the primitive lattice such that $[\Lambda_0]=[0,0,0,0]$, then $\Lambdabar_0$ is semisimple. Therefore, $\Gamma_\rho^\GSp$ has no edges. 
\end{proof}

\providecommand{\bysame}{\leavevmode\hbox to3em{\hrulefill}\thinspace}
\providecommand{\MR}{\relax\ifhmode\unskip\space\fi MR }
\providecommand{\MRhref}[2]{%
  \href{http://www.ams.org/mathscinet-getitem?mr=#1}{#2}
}
\providecommand{\href}[2]{#2}
  

\begin{thebibliography}{DKSW23b}



\bibitem[AB08]{AbramenkoBrown}
Peter Abramenko and Kenneth~S. Brown, \emph{Buildings}, Graduate Texts in
  Mathematics, vol. 248, Springer, New York, 2008, Theory and applications.
  \MR{2439729}

  \bibitem[BBK25]{berger2025klingeneisensteinseriescongruences}
 Tobias Berger, Jim Brown, and  Krzysztof Klosin, \emph{Klingen {E}isenstein congruences and modularity}, Canadian Journal of Mathematics, pp. 1–35. doi:10.4153/S0008414X25101612.


\bibitem[BC04]{BellaicheChenevierU3}
Jo\"el Bella\"iche and Ga\"etan Chenevier, \emph{Formes non temp\'er\'ees pour
  {$\rm U(3)$} et conjectures de {B}loch-{K}ato}, Ann. Sci. \'Ecole Norm. Sup.
  (4) \textbf{37} (2004), no.~4, 611--662. \MR{2097894}

\bibitem[BC09]{BellaicheChenevierFamilies}
\bysame, \emph{Families of {Galois}
  representations and {Selmer} groups}, Ast\'erisque, no. 324, Soci\'et\'e
  math\'ematique de France, 2009 (en). \MR{2656025}

\bibitem[Bel03]{BellaicheRibetslemmaPadova}
Jo\"el Bella\"iche, \emph{\`a{} propos d'un lemme de {R}ibet}, Rend. Sem. Mat.
  Univ. Padova \textbf{109} (2003), 45--62. \MR{1997986}

\bibitem[Bel09]{BellaicheRibet}
\bysame, \emph{Ribet’s lemma, generalizations, and
  pseudocharacters}, Two lectures at the Clay Mathematical Institute Summer
  School, 2009.

\bibitem[Bel19]{Bellaicheimages}
\bysame, \emph{Image of pseudo-representations and coefficients of
  modular forms modulo {$p$}}, Adv. Math. \textbf{353} (2019), 647--721.
  \MR{3989529}

\bibitem[BG06]{BellaicheGraftieaux}
Jo\"el Bella\"iche and Philippe Graftieaux, \emph{Repr\'esentations sur un
  anneau de valuation discr\`ete complet}, Math. Ann. \textbf{334} (2006),
  no.~3, 465--488. \MR{2207872}

\bibitem[BK21]{bergerklosinirred}
Tobias Berger and Krzysztof Klosin, \emph{Irreducibility of limits of {G}alois
  representations of {S}aito-{K}urokawa type}, Res. Number Theory \textbf{7}
  (2021), no.~3, Paper No. 41, 25. \MR{4269673}

\bibitem[Che14]{ChenevierDeterminants}
Ga\"{e}tan Chenevier, \emph{The {$p$}-adic analytic space of pseudocharacters
  of a profinite group and pseudorepresentations over arbitrary rings},
  Automorphic forms and {G}alois representations. {V}ol. 1, London Math. Soc.
  Lecture Note Ser., vol. 414, Cambridge Univ. Press, Cambridge, 2014,
  pp.~221--285. \MR{3444227}

\bibitem[DK23]{SamitKakde}
Samit Dasgupta and Mahesh Kakde, \emph{On the {B}rumer-{S}tark conjecture},
  Ann. of Math. (2) \textbf{197} (2023), no.~1, 289--388. \MR{4513146}

\bibitem[DKSW23a]{dasgupta2023brumerstarkconjecturez}
Samit Dasgupta, Mahesh Kakde, Jesse Silliman, and Jiuya Wang, \emph{The
  {B}rumer--{S}tark conjecture over $\mathbb{Z}$}, 2023.

\bibitem[DKSW23b]{dasgupta2023residually}
\bysame, \emph{The residually indistinguishable case of {R}ibet's method for
  $\mathrm{GL}_2$}, 2023.

\bibitem[Gar97]{PaulGarrett}
Paul Garrett, \emph{Buildings and classical groups}, Chapman \& Hall, London,
  1997. \MR{1449872}

\bibitem[GK96]{skewsymmetricgraphs}
Andrew~V. Goldberg and Alexander~V. Karzanov, \emph{Path problems in
  skew-symmetric graphs}, Combinatorica \textbf{16} (1996), no.~3, 353--382.
  \MR{1417346}

\bibitem[LS10]{LamSwanSymMod}
T.~Y. Lam and R.~G. Swan, \emph{Symplectic modules and von {N}eumann regular
  matrices over commutative rings}, Advances in ring theory, Trends Math.,
  Birkh\"auser/Springer Basel AG, Basel, 2010, pp.~213--227. \MR{2664673}

\bibitem[MQ26]{moakher2023symplecticdeterminantlawsinvariant}
Mohamed Moakher and Julian Quast, \emph{Symplectic determinant laws and
  invariant theory}, Journal of the Institute of Mathematics of Jussieu
  \textbf{25} (2026), no.~5, 2853–2918.

\bibitem[MW84]{Mazur--Wiles}
B.~Mazur and A.~Wiles, \emph{Class fields of abelian extensions of {${\bf
  Q}$}}, Invent. Math. \textbf{76} (1984), no.~2, 179--330. \MR{742853}

\bibitem[Rib76]{ribetunram}
Kenneth~A. Ribet, \emph{A modular construction of unramified {$p$}-extensions
  of {$\mathbb{Q}(\mu_{p})$}}, Invent. Math. \textbf{34} (1976), no.~3,
  151--162. \MR{419403}

\bibitem[Set08]{AlisonBuildings}
Alison Setyadi, \emph{The affine buildings of {SL}(n) and {S}p(n): {A}
  combinatorial perspective}, ProQuest LLC, Ann Arbor, MI, 2008, Thesis
  (Ph.D.)--Dartmouth College. \MR{2712416}

\bibitem[She07]{Shemanskesymplecticbuilding}
Thomas~R. Shemanske, \emph{The arithmetic and combinatorics of buildings for
  {${\rm Sp}_n$}}, Trans. Amer. Math. Soc. \textbf{359} (2007), no.~7,
  3409--3423. \MR{2299461}

\bibitem[SU14]{SkinnerUrban}
Christopher Skinner and Eric Urban, \emph{The {I}wasawa main conjectures for
  {$\rm GL_2$}}, Invent. Math. \textbf{195} (2014), no.~1, 1--277. \MR{3148103}

\bibitem[Urb01]{UrbanKlingen}
Eric Urban, \emph{Selmer groups and the {E}isenstein-{K}lingen ideal}, Duke
  Math. J. \textbf{106} (2001), no.~3, 485--525. \MR{1813234}

\bibitem[Wil90]{WilesTotallyReal}
A.~Wiles, \emph{The {I}wasawa conjecture for totally real fields}, Ann. of
  Math. (2) \textbf{131} (1990), no.~3, 493--540. \MR{1053488}

\end{thebibliography}
\end{document}